\documentclass[11pt]{amsart}

\usepackage{amsthm, amsmath}
\usepackage{amssymb}
\usepackage[all]{xy}
\usepackage{mathrsfs}
\usepackage[latin1]{inputenc}
\usepackage{graphicx}

\numberwithin{equation}{subsection}

\DeclareMathOperator{\Lie}{Lie}

\DeclareMathOperator{\defect}{def}

\theoremstyle{plain}

\newtheorem{theorem}{Theorem}[subsection]

\newtheorem{lemma}[theorem]{Lemma}

\newtheorem{cor}[theorem]{Corollary}

\newtheorem{defn}[theorem]{Definition}
\newtheorem{prop}[theorem]{Proposition}

\newtheorem{conj}[theorem]{Conjecture}

\theoremstyle{definition}
\newtheorem{rem}[theorem]{Remark}
\newtheorem{example}[theorem]{Example}

\begin{document}

\title[Affine Deligne-Lusztig varieties in affine flag varieties]
{Affine Deligne-Lusztig varieties\\in affine flag varieties}

\author[U. G\"{o}rtz]{Ulrich G\"{o}rtz}
\address{Ulrich G\"{o}rtz\\Institut f\"ur Experimentelle Mathematik\\Universit\"at Duisburg-Essen\\Ellernstr. 29\\45326 Essen\\Germany} \email{ulrich.goertz@uni-due.de}
\thanks{G\"{o}rtz was partially supported by a Heisenberg grant and by the
SFB/TR 45 ``Periods, Moduli Spaces and Arithmetic of Algebraic Varieties''
of the DFG (German Research Foundation)}

\author[T. J. Haines]{Thomas J. Haines}
\address{Thomas J. Haines\\Mathematics Department\\ University of Maryland\\ College Park, MD
20742-4015} \email{tjh@math.umd.edu}
\thanks{Haines was partially supported by NSF Grant FRG-0554254 and a Sloan Research
Fellowship}

\author[R. E. Kottwitz]{Robert E. Kottwitz}
\address{Robert E. Kottwitz\\Department of Mathematics\\ University of Chicago\\ 5734 University
Avenue\\ Chicago, Illinois 60637}
\email{kottwitz@math.uchicago.edu}
\thanks{Kottwitz was partially supported by NSF Grant DMS-0245639}

\author[D. C. Reuman]{Daniel C. Reuman}
\address{Daniel C. Reuman\\ Imperial College London\\ Silwood Park Campus\\ Buckhurst Road \\ Ascot \\
Berkshire \\ SL5 7PY \\ United Kingdom} \email{d.reuman@imperial.ac.uk}
\thanks{Reuman was partially supported by United States NSF grant DMS-0443803}


\begin{abstract}
This paper studies affine Deligne-Lusztig varieties in the affine flag manifold of a
split group. Among other things, it proves emptiness for certain of these varieties,
relates  some of them to those for Levi subgroups, and extends
previous conjectures concerning their dimensions. We generalize the superset method, an algorithmic approach to the questions of non-emptiness and dimension.  Our non-emptiness results apply equally well to the $p$-adic context and therefore relate to moduli of $p$-divisible groups and Shimura varieties with Iwahori level structure.
\end{abstract}

\maketitle

\section{Introduction}
\subsection{}
This paper, a continuation of \cite{GHKR}, investigates affine Deligne-Lusztig
varieties in the affine flag variety of a split connected reductive group $G$ over a finite field $k = \mathbb F_q$.  The Laurent series field $L= \overline{k}( (\varepsilon) )$, where $\overline{k}$ is an algebraic
closure of $k$, is endowed with a Frobenius automorphism $\sigma$, and we use the same symbol to denote the induced automorphism of $G(L)$.  By definition, the affine Deligne-Lusztig variety associated with $x$ in the extended affine Weyl
group $\widetilde{W} \cong I\backslash G(L)/I$ and $b\in G(L)$ is
\[
X_x(b) = \{ g\in G(L)/I;\ g^{-1}b\sigma(g)\in IxI\}.
\]
(See~\ref{notation} below for the notation used here.)  We are interested in determining the dimension of $X_x(b)$, and in finding a criterion for when $X_x(b)\ne\emptyset$. These questions are related to the geometric structure of the reduction of certain Shimura varieties with Iwahori level structure: on the special fiber of the Shimura variety we have, on one hand, the Newton stratification whose strata are indexed by certain $\sigma$-conjugacy classes $[b]\subseteq G(L)$, and on the other hand the Kottwitz-Rapoport stratification whose strata are indexed by certain elements of $\widetilde{W}$. The affine Deligne-Lusztig variety $X_x(b)$ is related to the intersection of the Newton stratum associated with $[b]$ and the Kottwitz-Rapoport stratum associated with $x$. See~\cite{GHKR}~5.10 and the survey papers of Rapoport~\cite{rapoport} and the second named author~\cite{haines}.

To provide some context we begin by  discussing  affine Deligne-Lusztig varieties
\[
X_\mu(b) = \{ g\in G(L)/K;\ g^{-1}b\sigma(g)\in K\varepsilon^\mu K\}
\]
in the affine Grassmannian $G(L)/K$.
It is known that $X_\mu(b)$ is non-empty if and only if Mazur's inequality is
satisfied, that is to say, if and only if the $\sigma$-conjugacy class $[b]$ of $b$
is less than or equal to $[\epsilon^\mu]$ in the natural partial order on the set
$B(G)$ of $\sigma$-conjugacy classes in $G(L)$. This was proved in two steps: the
problem was reduced \cite{kr03} to one on root systems, which was then solved for
classical split groups by C.~Lucarelli \cite{Lucarelli} and recently for all quasi-split groups
by Q.~Gashi \cite{Gashi}.

A conjectural formula for $\dim X_\mu(b)$ was put forward by Rapoport
\cite{rapoport}, who pointed out its similarity to a conjecture of Chai's
\cite{chai} on dimensions of Newton strata in Shimura varieties. In \cite{GHKR}
Rapoport's dimension conjecture was reduced to the superbasic case, which was then
solved by Viehmann \cite{V1}.

Now we return to affine Deligne-Lusztig varieties $X_x(b)$ in the affine flag
manifold.  For some years now a challenging  problem has been to ``explain'' the emptiness
pattern one sees in the figures in section~\ref{sec.examples}; see also \cite{Reuman2} and \cite{GHKR}. In other words, for
a given $b$, one wants to understand the set of $x \in \widetilde W$ for which
$X_x(b)$ is empty. Let us begin by discussing the simplest case, that in which $b=1$
and $x$ is \emph{shrunken}, by which we  mean that it lies in the union of the
shrunken Weyl chambers (see section~\ref{sec.examples} and \cite{GHKR}). Then Reuman \cite{Reuman2} observed
that a simple rule explained the emptiness pattern for $X_x(1)$ in types
$A_1$, $A_2$, and $C_2$, and  he conjectured  that the same might be true in
general.  Figure \ref{C2-1} in Section~\ref{sec.examples} illustrates how this simple rule depends on the elements $\eta_2(x)$ resp. $\eta_1(x)$ in $W$ labeling the ``big'' resp. ``small'' Weyl chambers which contain the alcove $x{\bf a}$.  (See section~\ref{relation_with_Reumans_conj} for the definitions of $\eta_1, \, \eta_2$ and Conjecture \ref{basic_Reu_conj} below for the precise rule.)  Computer calculations
\cite{GHKR} provided further evidence for the truth of Reuman's conjecture. However, although in the rank $2$ cases there is a simple geometric pattern in each strip between two adjacent Weyl chambers (see the figures in Section~\ref{sec.examples}), we do not have a closed formula in group-theoretic terms which is consistent with all higher rank examples we have computed when $x{\bf a}$ lies outside the shrunken Weyl chambers, and the emptiness there has remained mysterious.

In this paper, among other things, we give a precise conjecture describing the whole emptiness pattern 
for any basic $b$.  This is more
general in two ways: we no longer require that $b=1$ (though we do require that $b$
be basic), and we no longer restrict attention to shrunken $x$.  To do this we introduce the new notion 
of $P$-{\em alcove} for any semistandard parabolic subgroup $P = MN$ (see Definition \ref{def.Palcove}, sections~\ref{sec.statement_main_thm} and~\ref{acute_cones_section}).  Our Conjecture \ref{conj2} is as follows:

\begin{conj} \label{P-alc_conj}
Let $[b]$ be a basic $\sigma$-conjugacy class.  Then $X_x(b) \neq \emptyset$ if and only if, for every semistandard $P =MN$ for which $x{\bf a}$ is a $P$-alcove, $b$ is $\sigma$-conjugate to an element $b' \in M(L)$ and $x$ and $b'$ have the same image under the Kottwitz homomorphism $\eta_M: M(L) \rightarrow \Lambda_M$. 
\end{conj}
See section \ref{sec.revbg} for a review of $\eta_M$.   If $x{\bf a}$ is a $P$-alcove, then in particular $x\in \widetilde{W}_M$, the extended affine Weyl group of $M$, so that we can speak about $\eta_M(x)$.  
The condition $\eta_M(x) = \eta_M(b')$ means that $x$ 
and $b'$ lie in the same connected component of the $\overline{k}$-ind-scheme $M(L)$.  Computer calculations
support this conjecture, and for shrunken $x$ we show (see Proposition
\ref{prop.crcon}) that the new conjecture reduces to Reuman's. 
We prove (see Corollary \ref{ConsequBasic}) one direction of this new conjecture, namely:

\begin{theorem} Let $[b]$ be basic.  Then $X_x(b)$ is empty when Conjecture \ref{P-alc_conj} predicts it to be.  
\end{theorem}
It remains a challenging problem to prove that non-emptiness occurs when predicted.

In fact  Proposition \ref{nec_cond_prop} proves the emptiness of certain $X_x(b)$
even when $b$ is not basic. However, in the non-basic case, there is a second cause
for emptiness, stemming from Mazur's inequality. One might hope that these are the
only two causes for emptiness. This is slightly too naive. Mazur's inequality works
perfectly for $G(\mathfrak o)$-double cosets, but not for Iwahori double cosets, and
would have to be improved slightly (in the Iwahori case)
before it could be applied to
give an optimal emptiness criterion.  Although we do not yet know how
to formulate Mazur's inequalities
in the Iwahori case, in section \ref{reduction_to_basic_section} we are able to
describe the information they should carry, whatever they end up being.

We now turn to the dimensions of non-empty affine Deligne-Lusztig varieties in the
affine flag manifold. In \cite{GHKR} we formulated two conjectures of this kind, and
here we will extend both of them (in a way that is supported by computer evidence).
For basic $b$, we have

\begin{conj} \label{basic_Reu_conj}[Conjecture~\ref{conj3}(a)]
Let $[b]$ be a basic $\sigma$-conjugacy class.
Suppose $x\in \widetilde{W}$ lies in the shrunken Weyl chambers.
Then $X_x(b)\ne \emptyset$ if and only if
\[
\eta_G(x)=\eta_G(b), \text{ and }
\eta_2(x)^{-1}\eta_1(x)\eta_2(x) \in W\setminus \bigcup_{T\subsetneq S} W_T,
\]
and in this case
\[
\dim X_x(b) = \frac{1}{2}\left( \ell(x)
+\ell(\eta_2(x)^{-1}\eta_1(x)\eta_2(x)) - {\rm def}_G(b) \right).
\]
\end{conj}
Here $\defect_G(b)$ denotes the defect of $b$ (see section~\ref{relation_with_Reumans_conj}).  
This extends Conjecture 7.2.2 of \cite{GHKR} from $b=1$ to all basic $b$. For an illustration in the case of $G=GSp_4$ (where the conjecture can be checked as in \cite{Reuman2}), see section~\ref{sec.examples}.

Conjecture
\ref{conj3}(b) extends Conjecture 7.5.1 of \cite{GHKR} from translation elements
$b=\epsilon^\nu$ to all $b$. For this we need the following notation: $b_{\rm b}$
will denote a representative of the unique basic  $\sigma$-conjugacy class whose image in
$\Lambda_G$ is the same as that of $b$. (Equivalently, $[b_{\rm b}]$ is at the
bottom of the connected component of $[b]$ in the poset $B(G)$.)  In this second
conjecture, it is the difference of the dimensions of $X_x(b)$ and $X_x(b_{\rm b})$
that is predicted. It is not required that
$x$  be shrunken, but $X_x(b)$ and $X_x(b_{\rm b})$ are required to
be non-empty, and the length of $x$ is required to be big enough. In the conjecture
we phrase this last condition rather crudely as $\ell(x) \ge N_b$ for some
(unspecified) constant $N_b$ that depends on $b$. However the evidence of  computer
calculations suggests that for fixed $b$, having $x$ such that $X_x(b)$ and
$X_x(b_{\rm b})$ are both non-empty is almost (but not quite!) enough to make our
prediction valid for  $x$. It would be very interesting to understand this
phenomenon better, though some insight into it is already provided by Beazley's work
on Newton strata for $SL(3)$ \cite{Beazley}. In addition, when $\ell(x) \ge N_b$, we
conjecture that the non-emptiness of $X_x(b)$ is equivalent to that of $X_x(b_{\rm
b})$.

The main theorem of this paper is a version of the Hodge-Newton decomposition which relates certain affine Deligne-Lusztig varieties for the group $G$ to affine Deligne-Lusztig varieties for a Levi subgroup $M$:

\begin{theorem}[Theorem \ref{HN}, Corollary \ref{sigma_quotient}]
Suppose $P= MN$ is semistandard and $x{\bf a}$ is a $P$-alcove.
\begin{enumerate}
\item[(a)]  The natural map $B(M) \to B(G)$ restricts to a bijection $B(M)_x \to B(G)_x$, where $B(G)_x$ is the subset of $B(G)$ consisting of $[b]$ for which $X^G_x(b)$ is non-empty. In particular, 
if $X^G_x(b) \neq \emptyset$, then $[b]$ meets $M(L)$.
\item[(b)]  Suppose $b \in M(L)$.  Then the canonical closed immersion $X^M_x(b) \hookrightarrow X^G_x(b)$ induces a bijection
$$
J^M_b \backslash X^M_x(b) ~~ \widetilde{\rightarrow} ~~ J^G_b \backslash X^G_x(b),
$$
where $J^G_b$ denotes the $\sigma$-centralizer of $b$ in $G(L)$ (see section \ref{sec.statement_main_thm}).
\end{enumerate}
\end{theorem}

The second part of this theorem can be proved using the techniques of \cite{Kottwitz-HN}, but it seems unlikely that 
the same is true of the first part. In any case we use a different method, obtaining both parts of the theorem 
as a consequence of the following key   
 result (Theorem \ref{mainthm}),   whose precise relation to the Hodge-Newton decomposition is clarified by the commutative diagram 
 \eqref{diagram}. 

\begin{theorem} 
For any semistandard parabolic subgroup $P=MN$ and any $P$-alcove $x\mathbf a$, every
element of $IxI$ is $\sigma$-conjugate under $I$ to an element of $I_MxI_M$, where $I_M := M \cap I$.
\end{theorem}

It is striking that the notion of $P$-alcove, discovered in the attempt to
understand the entire emptiness pattern for the $X_x(b)$ when $b$ is basic, is
also precisely the notion needed for our Hodge-Newton decomposition.

In sections~\ref{dim_theory_section}--\ref{sec.supset} we consider the questions of non-emptiness and dimensions of affine Deligne-Lusztig varieties from an algorithmic point of view.  The following summarizes Theorem \ref{dim_alg} and Corollary \ref{cor.superset_method}:

\begin{theorem} \label{algor}
There are algorithms, expressed in terms of foldings in the Bruhat-Tits building of $G(L)$, for determining the non-emptiness and dimension of $X_x(b)$.  
\end{theorem}

These algorithms were used to produce the data that led to and supported our conjectures.  The results of these sections imply in particular that the non-emptiness is equivalent in the function field and the $p$-adic case (Corollary~\ref{functionfield_vs_padic}). While this was certainly expected to hold, to the best of our knowledge no proof was known before. This equivalence is used by Viehmann~\cite{V3} to investigate closure relations for Ekedahl-Oort strata in certain Shimura varieties; our results enable her to carry over results from the function field case, thus avoiding the heavy machinery of Zink's displays. It seems plausible that the algorithmic description of Theorem~\ref{dim_alg} can also be used to show that the dimensions in the function field case and the $p$-adic case coincide, once a good notion of dimension has been defined in the latter case.

In section~\ref{sec.supset} we extend Reuman's superset method \cite{Reuman2} from
$b=1$ to general $b$. To that end we introduce (see Definition \ref{def.fa})
the notion of \emph{fundamental alcove} $y\mathbf a$. We show that for each
$\sigma$-conjugacy class $[b]$ there exists a fundamental alcove $y\mathbf a$
such that the whole double coset $IyI$ is contained in $[b]$. We then explain
why this allows one to use a superset method to analyze the emptiness of
$X_x(b)$ for any $x$.

In addition we introduce, in Chapter~\ref{superset_section}, a generalization
of the superset method. The superset method is based on $I$-orbits in the
affine flag manifold $X$. It depends on the choice of a suitable representative
for $b$, whose existence is proved in Chapter~\ref{sec.supset}, as mentioned
above. On the other hand, \cite{GHKR} used orbits of $U(L)$, where $U$ is the
unipotent radical of a Borel subgroup containing our standard split maximal
torus $A$.  The generalized superset method interpolates between these two
extremes, being based on orbits of $I_MN(L)$ on $X$, where $P=MN$ is a standard
parabolic subgroup of $G$. Theorem \ref{dim_alg} and the discussion preceding
it explain how the generalized superset method can be used to study dimensions
of affine Deligne-Lusztig varieties.

For any standard parabolic subgroup $P=MN$ and any basic $b \in M(L)$
Proposition \ref{prop.&&&} gives a formula for the dimension of $X_x(b)$ in terms of
dimensions of affine Deligne-Lusztig varieties for $M$ as well as intersections of
$I$-orbits and $N'(L)$-orbits for certain Weyl group conjugates $N'$ of $N$. This generalizes
Theorem 6.3.1 of \cite{GHKR} and is also analogous to Proposition 5.6.1 of
\cite{GHKR}, but with the affine Grassmannian replaced by the affine flag manifold.

\subsection*{Acknowledgments} The first and second named authors thank the American Institute of Mathematics for the
invitation to the workshop on Buildings and Combinatorial Representation
Theory which provided an opportunity to work on questions related to this
paper. The second and third named authors  thank the University of Bonn and
the Max-Planck-Institut f\"ur Mathematik Bonn
for providing other, equally valuable opportunities of this kind.

We are very grateful to Xuhua He and Eva Viehmann for their helpful remarks on the manuscript.

\subsection{Notation}\label{notation}  We follow the notation of \cite{GHKR}, for the most part.
 Let $k$ be a finite field with $q$ elements,
and let $\overline k$ be an algebraic
closure of $k$.  We consider the field $L:=\overline k((\epsilon))$
and its subfield
$F:=k((\epsilon))$. We write $\sigma:x\mapsto x^q$ for the Frobenius
automorphism of
$\overline k/k$, and we also regard $\sigma$ as an automorphism of $L/F$ in the
usual way, so that $\sigma(\sum a_n\epsilon^n)=\sum \sigma(a_n)\epsilon^n$.
We write
$\mathfrak o$  for the valuation ring $\overline k[[\epsilon]]$ of~$L$.

Let $G$ be a split connected reductive group over $k$, and let $A$ be a split
maximal torus of~$G$. Write $R$ for the set of roots of $A$ in $G$. Put $\mathfrak
a:=X_*(A)_
\mathbb R$. Write
$W$ for the Weyl group of
$A$ in
$G$.   Fix a Borel subgroup
$B=AU$ containing $A$ with unipotent radical~$U$, and write $R^+$ for the
corresponding set of positive roots, that is, those occurring in $U$.
We denote by $\rho$  the half-sum of the positive roots.
For $\lambda \in
X_*(A)$ we write
$\epsilon^\lambda$ for the element of $A(F)$ obtained as the image of
$\epsilon
\in \mathbb G_m(F)$ under the homomorphism $\lambda:\mathbb G_m \to A$.

Let $C_0$ denote the dominant Weyl
chamber, which by definition is the set of
$ x \in\mathfrak a$ such that  $ \langle \alpha, x \rangle > 0$ for all
$\alpha \in R^+$.    We denote by $\mathbf a$ the
unique alcove in the dominant Weyl chamber whose closure contains the
origin, and call it the base alcove. As Iwahori subgroup $I$ we choose the
one fixing the base alcove $\mathbf a$;  $I$ is then the inverse image of the opposite
Borel group of $B$ under the projection $K:=G(\mathfrak o) \longrightarrow
G(\overline{k})$. The opposite Borel arises  here due to our convention that
$\epsilon^\lambda$ acts on the standard apartment
$\mathfrak a$ by translation by $\lambda$ (rather than by translation by
 the negative of $\lambda$), so that the stabilizer in $G(L)$ of $\lambda
\in X_*(A) \subset \mathfrak a$ is $\epsilon^\lambda K\epsilon^{-\lambda}$. With this convention the
Lie algebra of the Iwahori subgroup stabilizing an alcove $\mathbf b$ in the
standard apartment is made up of  affine root spaces $\epsilon^j\mathfrak g_\alpha$
for all pairs
$(\alpha,j)$ such that $\alpha -j \le 0$ on $\mathbf b$ (with $\mathfrak g_\alpha$
denoting the root subspace corresponding to $\alpha$).

We will often think of alcoves in a slightly different way.  Let $\Lambda_G$ denote the quotient of $X_*(A)$ by the coroot lattice.  The apartment $\mathcal A$ corresponding to our
fixed maximal torus $A$ can be decomposed as a product $\mathcal A = \mathcal A_{\rm der} \times V_G$, where $V_G := \Lambda_G \otimes \mathbb R$ and where $\mathcal A_{\rm der}$ is the apartment corresponding to $A_{\rm der} := G_{\rm der} \cap A$ in the building for $G_{\rm der}$.  By an {\em extended alcove} we mean a subset of the apartment $\mathcal A$ of the form $\mathbf b \times c$, where $\mathbf b$ is an alcove in $\mathcal A_{\rm der}$ and $c \in \Lambda_G$.  Clearly each extended alcove determines a unique alcove in the usual sense, but not conversely.  However, in the sequel we will often use the terms interchangeably, leaving
context to determine what is meant.  In particular, we often write ${\bf a}$ in place of ${\bf a} \times 0$.

We denote by $\widetilde{W}$ the extended affine Weyl group $X_*(A) \rtimes W$ of $G$.  Then $\widetilde{W}$ acts transitively on the set of all alcoves in $\mathfrak a$, and simply transitively on the set of all extended alcoves.  Let $\Omega = \Omega_{\bf a}$ denote the stabilizer of ${\bf a}$ when it is viewed as an alcove in the usual (non-extended) sense.  We can write an extended (resp. non-extended) alcove in the form $x{\bf a}$ for a unique element $x \in \widetilde W$ (resp. $x \in \widetilde W/\Omega$).  Of course, this is just another way of saying that we can think of extended alcoves simply as elements of $\widetilde{W}$. Note that we can also describe $\widetilde{W}$ as the quotient $N_GA(L)/A(\mathfrak o)$. For $x\in \widetilde{W}$, we write ${^x} I = \dot{x} I \dot{x}^{-1}$, were $\dot{x}\in N_GA(L)$ is a lift of $x$. It is clear that the result is independent of the choice of lift.

 As usual a standard parabolic subgroup is one containing $B$, and  a semistandard
parabolic subgroup  is one containing $A$. Similarly, a semistandard Levi subgroup is
one containing $A$, and a standard Levi subgroup is the unique semistandard Levi component of a
standard parabolic subgroup. Whenever we write $P=MN$ for a semistandard parabolic subgroup, we take this to mean that $M$ is its semistandard Levi component, and that $N$ is its unipotent radical. Given a semistandard Levi subgroup $M$ of $G$ we write $\mathcal
P(M)$ for the set of parabolic subgroups of $G$ admitting $M$ as Levi component. For
$P \in
\mathcal P(M)$ we denote by $\overline P=M\overline N \in \mathcal P(M)$ the parabolic subgroup
opposite to $P$, i.~e.~$\overline N$ is the unipotent radical of $\overline P$. We write $R_N$ for the set of roots of $A$ in $N$. We denote by
$I_M$, $I_N$, $I_{\overline N}$ the intersections of $I$ with $M$, $N$, $\overline N$
respectively; one then has the Iwahori decomposition $I=I_NI_MI_{\overline N}$.

Recall that for $x \in \widetilde W$ and $b \in G(L)$ the affine
Deligne-Lusztig variety
$X_x(b)$ is defined by
\[
X_x(b):=\{ g \in G(L)/I: g^{-1}b\sigma(g) \in IxI\}.
\]

In the sequel we often abuse notation and use the symbols $G,P,M,N$ to
denote the corresponding objects over $L$.

Let $b\in G(L)$.
We denote by $[b]$ the $\sigma$-conjugacy class of $b$ inside $G(L)$:
\[ [b] = \{ g^{-1}b\sigma(g);\ g\in G(L) \}, \]
and for a subgroup $H\subseteq G(L)$ we write
\[ [b]_H := \{h^{-1}b\sigma(h);\ h\in H  \} \subseteq G(L) \]
for the $\sigma$-conjugacy class of $b$ under $H$.
Further notation  relevant to $B(G)$ such as $\eta_G$  will be explained
in section \ref{sec.revbg}.

Finally we note that ${}^xI$ will be used as an abbreviation for $xIx^{-1}$.
We use the symbols $\subset$ and $\subseteq$ interchangeably with the
meaning ``not necessarily strict inclusion''.

\section{Statement of the main Theorem}\label{sec.statement_main_thm}
\subsection{} 
Let $\alpha \in R$. We identify the root group $U_\alpha$ with
the additive group $\mathbb G_a$ over $k$, which then allows us to identify
$U_\alpha(L)\cap K$ with
$\mathfrak o$. The root $\alpha$ induces a partial order $\ge_\alpha$ on
the set of (extended) alcoves in the standard apartment as follows: given an alcove $\mathbf b$,
write it as
$x{\bf a}$ for
$x \in \widetilde{W}$.  Let
$k(\alpha, \mathbf b)\in\mathbb Z$ such that $U_\alpha(L) \cap \, {^x}I =
\epsilon^{k(\alpha, \mathbf b)}\mathfrak o$.  In other words, $k(\alpha,
\mathbf b)$ is the unique integer $k$ such that
$\mathbf b$ lies in the region between the affine root hyperplanes
$H_{\alpha, k}=\{x\in X_*(A)_{\mathbb R};\ \langle \alpha, x \rangle = k\}$
and $H_{\alpha, k-1}$. This description shows immediately that $k(\alpha,
\mathbf b) + k(-\alpha,\mathbf b) = 1$.
(For instance, we have $k(\alpha,
\mathbf a) = 1$ if $\alpha>0$ and $k(\alpha, \mathbf a) = 0$ if
$\alpha<0$. This reflects the fact that the fixer $I$ of ${\bf a}$ is
the inverse image of the opposite Borel $\overline{B}$ under the projection
$G(\mathfrak o)\rightarrow G(\overline k)$.)
We define
\[ \mathbf b_1 \ge_\alpha \mathbf b_2 :\Longleftrightarrow k(\alpha,
\mathbf b_1) \ge k(\alpha, \mathbf b_2). \]
This is a partial order in the weak sense: $ \mathbf b_1 \ge_\alpha \mathbf b_2$ and
$ \mathbf b_2 \ge_\alpha \mathbf b_1$ does not imply that $ \mathbf b_1
= \mathbf b_2$. We also define
\[ \mathbf b_1 >_\alpha \mathbf b_2 :\Longleftrightarrow k(\alpha,
\mathbf b_1) > k(\alpha, \mathbf b_2). \]

\begin{defn}\label{def.Palcove}
Let $P=MN$ be a semistandard parabolic subgroup.
Let $x\in\widetilde{W}$.  We say $x{\bf a}$  is a \emph{$P$-alcove}, if
\begin{enumerate}
\item[(1)] $x \in \widetilde{W}_M$, and
\item[(2)] $\forall \alpha \in R_N$, $x{\bf a} \geq_\alpha {\bf a}$.
\end{enumerate}
We say $x{\bf a}$ is a \emph{strict} $P$-alcove if instead of (2) we have
\begin{enumerate}
\item[($2'$)] $\forall \alpha \in R_N$, $x{\bf a} >_\alpha {\bf a}$.
\end{enumerate}
\end{defn}

Note that condition (2) depends only on the image of $x$ in $\widetilde{W}/\Omega$; however, condition (1) depends on $x$ itself. It is important to work with extended alcoves here. One could argue that the above definition is rather about elements of the extended affine Weyl group than about extended alcoves, but the term \emph{$P$-alcove} seemed most convenient anyway.

By the definition of the partial order $\geq_\alpha$, the condition (2) is
equivalent to
\begin{equation} \label{condition(2)}
\forall \alpha \in R_N,\quad  U_\alpha \cap \, ^xI \subseteq
U_\alpha \cap I,
\end{equation}
or, likewise, to
\begin{equation}
\forall \alpha \in R_N,\quad U_{-\alpha} \cap \, ^xI \supseteq U_{-\alpha} \cap I
\end{equation}
and under our assumption that $x \in \widetilde{W}_M$,
these in turn are equivalent to the condition
\begin{equation} \label{condition(2A)}
^x(N \cap I) \subseteq N \cap I \quad\text{or, equivalently to}\quad
{^x}(\overline{N} \cap I) \supseteq \overline{N} \cap I.
\end{equation}
(And condition ($2'$) is equivalent to (\ref{condition(2)}) with the
inclusions replaced by strict inclusions.)  Indeed, noting that conjugation
by $x = \epsilon^\lambda w$ permutes the subgroups $U_\alpha$ with $\alpha
\in R_N$, it is easy to see from the (Iwahori) factorization
\begin{equation} \label{I_cap_N_fact}
N \cap I = \prod_{\alpha \in R_N} U_\alpha \cap I,
\end{equation}
that (\ref{condition(2)}) is equivalent to (\ref{condition(2A)}).
For a fixed semistandard parabolic subgroup $P = MN$, the set of alcoves $x{\bf a}$
which satisfy (\ref{condition(2)}) forms a union of ``acute cones of alcoves'' in the
sense of \cite{Haines-Ngo}. We shall explain this in section
\ref{acute_cones_section} below.

\medskip

Our key result concerns the map
\begin{align*}
\phi: I \times I_M x I_M &\rightarrow I x I \\
(i,m) &\mapsto im \, \sigma(i)^{-1}.
\end{align*}
There is a left action of $I_M$ on $I \times I_M x I_M$ given by $i_M(i,m)
= (ii_M^{-1}, i_M m \sigma(i_M)^{-1})$, for $i_M \in I_M$, $i \in I$ and $m
\in I_M x I_M$.  Let us denote by $I \times^{I_M} I_M x I_M$ the quotient
of $I \times I_M x I_M$ by this action of $I_M$.   Denote by $[i,m]$ the
equivalence class of $(i,m) \in I \times I_M x I_M$.  The map $\phi$
obviously factors through $I \times^{I_M} I_M x I_M$.  We can now state the key result which enables us to prove the Hodge-Newton decomposition.

\begin{theorem} \label{mainthm}
Suppose $P = MN$ is a semistandard parabolic subgroup, and $x\bf a$ is a $P$-alcove.
Then the map
$$
\phi: I \times^{I_M} I_M x I_M \rightarrow IxI
$$
induced by $(i,m) \mapsto im \sigma(i)^{-1}$, is surjective.  If $x{\bf a}$ is a strict $P$-alcove, then $\phi$ is injective.  In general, $\phi$ is not injective, but if $[i,m]$ and $[i',m']$ belong to the same fiber of $\phi$, the elements $m$ and $m'$ are $\sigma$-conjugate by an element of $I_M$.
\end{theorem}

This theorem was partially inspired by Labesse's study of the ``elementary functions'' he introduced in \cite{La}.

Let us mention a few consequences.  First, consider the quotient $IxI/_\sigma \, I$,
where the action of $I$ on $IxI$ is given by $\sigma$-conjugation.  We also can form
in a parallel manner the quotient $I_M x I_M/_\sigma \,I_M$.  Further, let $B(G)_x$
denote the set of $\sigma$-conjugacy classes $[b]$ in $G(L)$ which meet $IxI$. We
note that for $G=SL_3$ all of the sets $B(G)_x$ have been determined explicitly by
Beazley
\cite{Beazley}.

\begin{cor}\label{sigma_quotient} Suppose $P = MN$ is semistandard, and
$x{\bf a}$ is a $P$-alcove.  Then the following statements hold.
\begin{enumerate}
\item[(a)] The inclusion $I_M x I_M \hookrightarrow IxI$ induces a bijection
\begin{equation*}
I_M x I_M/_\sigma \,I_M  ~~ \widetilde{\rightarrow} ~~ IxI/_\sigma \,I.
\end{equation*}
\item[(b)] The canonical map $\iota: B(M)_x \rightarrow B(G)_x$ is bijective.
\end{enumerate}

\end{cor}
Part (a) follows directly from Theorem \ref{mainthm}.  Indeed, the surjectivity of $\phi$ implies the surjectivity of $I_M x I_M /_\sigma \,I_M \rightarrow IxI/_\sigma \, I$.  As for the injectivity of the latter, note that if $i \in I$ and $m,m' \in I_MxI_M$ satisfy $im\sigma(i)^{-1} = m'$, then $[i,m]$ and $[1,m']$ belong to the same fiber of
$\phi$.  As for part (b), we will derive it from part (a) in section
\ref{some_proofs}. (In fact the surjectivity in part (b) follows easily from the surjectivity in Theorem \ref{mainthm}.)

Another consequence is our main theorem, a version of the Hodge-Newton decomposition, given
in Theorem \ref{HN} below.  For affine Deligne-Lusztig varieties in the
affine Grassmannian of a split group, the analogous Hodge-Newton decomposition was
 proved under unnecessarily strict hypotheses in \cite{Kottwitz-HN} and in the
general case by Viehmann \cite[Theorem 1]{V2} (see also
  Mantovan-Viehmann \cite{MV} for the case of unramified groups).  To
state this we need to fix a standard parabolic subgroup $P = MN$ and an element $b \in
M(L)$.  Let $K_M = M
\cap K$, where $K$, as usual, denotes $G(\mathfrak o)$.  For a $G$-dominant coweight
$\mu
\in X_*(A)$, the
$\sigma$-centralizer
$J^G_b:=\{g \in G(L):g^{-1}b\sigma(g)=b\}$ of $b$  acts naturally on the
affine Deligne-Lusztig variety
$X^G_\mu(b) \subset G(L)/K$ defined to be
$$
X^G_\mu(b) := \{ gK \in G(L)/K ~ | ~ g^{-1}b\sigma(g) \in K\epsilon^\mu K \}.
$$
Also, $J^M_b$ acts on $X^M_\mu(b) \subset M(L)/K_M$.  Now the Hodge-Newton decomposition under discussion asserts the following: suppose that the Newton point $\overline{\nu}^M_b \in X_*(A)_\mathbb R$ is $G$-dominant, and that $\eta_M(b) = \mu$ in $\Lambda_M$.  Then the canonical closed immersion $X^M_\mu(b) \hookrightarrow X^G_\mu(b)$ induces a bijection
$$
J^M_b \backslash X^M_\mu(b) ~ \widetilde{\rightarrow} ~ J^G_b \backslash X^G_\mu(b).
$$
Of course if we impose the stricter condition
that $\langle \alpha, \overline{\nu}_b^M \rangle > 0$ for all $\alpha \in R_N$, then
$J^M_b = J^G_b$ and so we get the stronger conclusion $X^M_\mu(b) \cong X^G_\mu(b)$,
yielding what is normally known as the Hodge-Newton decomposition in this context.
The version with the weaker condition is essentially  a result of Viehmann, who
formulates it somewhat differently
\cite[Theorem 2]{V2}, in a way that brings out a dichotomy occurring when
$G$ is simple.

In the affine flag variety, it still makes sense to ask how $X^G_x(b)$ and $X^M_x(b)$ are related, for $x \in \widetilde{W}_M$ and $b \in M(L)$.  Our Hodge-Newton decomposition below provides some information in this direction.

\begin{theorem} \label{HN}
Suppose $P= MN$ is semistandard and $x{\bf a}$ is a $P$-alcove.
\begin{enumerate}
\item[(a)]  If $X^G_x(b) \neq \emptyset$, then $[b]$ meets $M(L)$.
\item[(b)]  Suppose $b \in M(L)$.  Then the canonical closed immersion $X^M_x(b) \hookrightarrow X^G_x(b)$ induces a bijection
$$
J^M_b \backslash X^M_x(b) ~~ \widetilde{\rightarrow} ~~ J^G_b \backslash X^G_x(b).
$$
\end{enumerate}
 \end{theorem}

Note that part (b) implies that if $x{\bf a}$ is a $P$-alcove, then for every $b \in M(L)$, we have $X^G_x(b) = \emptyset $ if and only if $X^M_x(b) = \emptyset$.  We will prove Theorem \ref{HN} in section \ref{some_proofs}  and then derive some further consequences relating to emptiness/non-emptiness of $X^G_x(b)$, in section \ref{consequences}.

\section{$P$-alcoves and acute cones of alcoves} \label{acute_cones_section}
\subsection{}
Let $P = MN$ be a fixed semistandard parabolic subgroup.  The aim of this section is to link the new notion of $P$-alcove to the notion of acute cones, and to help the reader visualize the set of $P$-alcoves.  Let $\mathfrak P$ denote the set of alcoves $x{\bf a}$ which satisfy the inequalities $x{\bf a} \geq_\alpha {\bf a}$ for all $\alpha \in R_N$.

For each element $w \in W$, we recall the notion of {\em acute cone} of alcoves $C({\bf a},w)$, following \cite{Haines-Ngo}.  Given an affine hyperplane
$H = H_{\alpha, k} = H_{-\alpha,-k}$, we assume $\alpha$ has the sign such
that $\alpha \in w(R^+)$, i.~e.~such that $\alpha$ is a positive root with
respect to ${^w}B$.  Then define the $w$-{\em positive} half space
$$H^{w+} = \{ v \in X_*(A)_\mathbb R ~ : ~ \langle \alpha, v \rangle > k \}.$$
Let $H^{w-}$ denote the other half-space.

Then the acute cone of alcoves $C({\bf a},w)$ is defined to be the set of alcoves $x{\bf a}$ such that some (equivalently, every) minimal gallery joining ${\bf a}$ to $x{\bf a}$ is in the $w$-direction.  By definition, a gallery ${\bf a}_1, \dots, {\bf a}_l$ is {\em in the $w$-direction} if for each crossing ${\bf a}_{i-1} |_H {\bf a}_i$, the alcove ${\bf a}_{i-1}$ belongs to $H^{w-}$ and ${\bf a}_i$ belongs to $H^{w+}$.  By loc.~cit.~Lemma 5.8, the acute cone $C({\bf a},w)$ is an intersection of half-spaces:
$$
C({\bf a}, w) = \bigcap_{{\bf a} \subset H^{w+}} H^{w+}.
$$

\begin{figure}
\includegraphics[width=12cm]{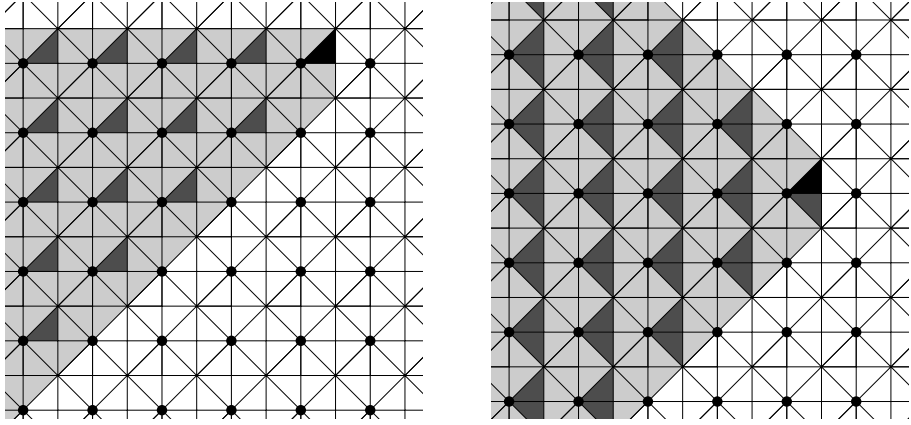}
\caption{The figure illustrates the notion of $P$-alcove for $G$ of type
$C_2$. On the left, $P = {^{w_0}}B$, where $w_0$ is the longest element in
$W$. On the right, $P = {^{s_1s_2s_1}P'}$ where $P'$ is the standard
parabolic $B \cup Bs_2B$. In both cases, the black alcove is the base
alcove, the region $\mathfrak P$ is in light gray, and the $P$-alcoves are
shown in dark gray.}
\end{figure}

\begin{figure}
\includegraphics[width=12cm]{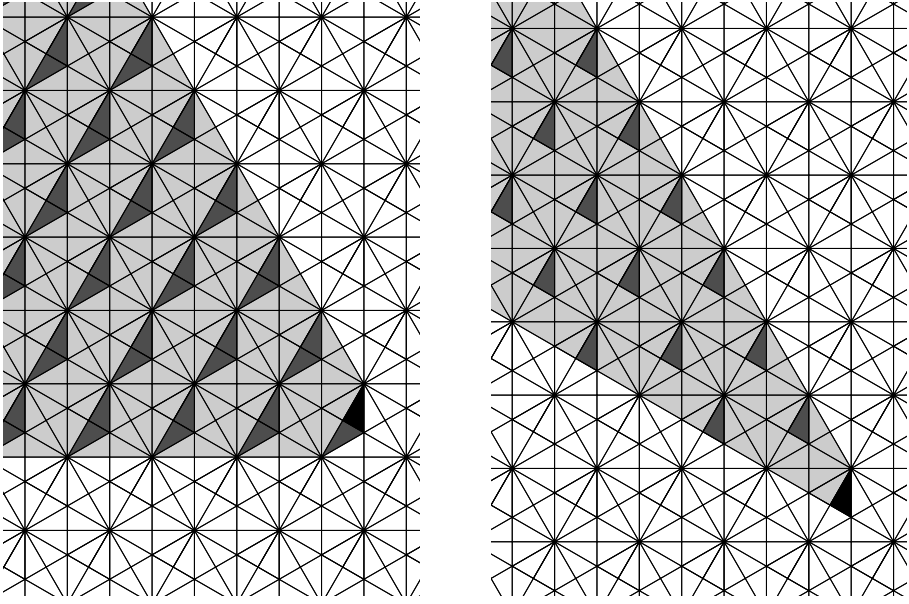}
\caption{
    This figure shows $P$-alcoves for $G$ of type $G_2$. On the left, $P =
    {^{s_1s_2s_1}}(B \cup Bs_2B)$, on the right, $P={^{s_2s_1s_2s_1}}B$.
    }
\end{figure}

\begin{prop} \label{acute_cone_prop}
The set of alcoves $\mathfrak P$ is the following union of acute cones of alcoves
\begin{equation} \label{mathfrakP=union}
\mathfrak P = \bigcup_{w \, : \, P \supseteq \, ^wB} C({\bf a},w).\end{equation}
\end{prop}

\begin{proof}
For any root $\alpha \in R$ and $k \in \mathbb Z$, let $H^+_{\alpha,k}$ denote the unique half-space for $H_{\alpha,k}$ which contains the base alcove ${\bf a}$.  Note that for any $\alpha \in R$ and $w \in W$, we have
\begin{equation} \label{H+}
H^+_{\alpha, k(\alpha, {\bf a})-1} = \begin{cases} H^{w+}_{\alpha, k(\alpha,{\bf a})-1}, \,\,\, \mbox{if $\alpha \in w(R^+)$} \\
H^{w-}_{\alpha, k(\alpha, {\bf a})-1}, \,\,\, \mbox{if $\alpha \in w(R^-)$.}
\end{cases}
\end{equation}

Now suppose $w \in W$ satisfies $P \supseteq \, ^wB$, or in other words $N \subseteq \, ^wU$, or equivalently,
$R_N \subseteq w(R^+)$.  Then we see using (\ref{H+}) that
\begin{equation*}
C({\bf a},w) = \bigcap_{\alpha \in w(R^+)} H^{w+}_{\alpha, k(\alpha, {\bf a})-1}
= \bigcap_{\alpha \in w(R^+)} H^{+}_{\alpha, k(\alpha, {\bf a})-1},
\end{equation*}
so the union on the right hand side of (\ref{mathfrakP=union}) is
\begin{equation}
\bigcup_{w\, :\, R_N \subseteq w(R^+)}\ \, \bigcap_{\alpha \in w(R^+)}
H^{+}_{\alpha, k(\alpha, {\bf a})-1}
\end{equation}
and in particular is contained in $\bigcap_{\alpha \in R_N} H^+_{\alpha,
k(\alpha, {\bf a})-1} = \mathfrak P$.

For the opposite inclusion, we set
$$\mathscr U = \bigcup_{w \, : \, R_N \subseteq w(R^+)} C({\bf a},w).$$
We will prove the implication
\begin{equation} \label{opp_incl}
x{\bf a} \notin \mathscr U \Longrightarrow x{\bf a} \notin \mathfrak P
\end{equation}
by induction on the length $\ell$ of a minimal gallery ${\bf a} = {\bf a}_0, {\bf a}_1, \dots, {\bf a}_\ell = x{\bf a}$.  If $\ell = 0$, there is nothing to show, so we assume that $\ell > 0$ and that the implication holds for $y{\bf a} :=
{\bf a}_{\ell-1}$.

Assume $x{\bf a} \notin \mathscr U$.  There are two cases to consider.  If $y{\bf a} \notin \mathscr U$, then by induction $y{\bf a} \notin \mathfrak P$.  This means that $y{\bf a}$ and ${\bf a}$ are on opposite sides of a  hyperplane $H_{\alpha, k(\alpha, {\bf a})-1}$ for some $\alpha \in R_N$.  The same then holds for $x{\bf a}$, which shows that $x{\bf a} \notin \mathfrak P$.

Otherwise, $y{\bf a} \in \mathscr U$, so that $y{\bf a}$ belongs to some $C({\bf a}, w)$ with $R_N \subseteq w(R^+)$.  Let $H= H_{\beta,m}$ be the wall separating $y{\bf a}$ and $x{\bf a}$.  Since $x{\bf a} \notin C({\bf a}, w)$ and $s_{\beta,m}x{\bf a} \in C({\bf a},w)$, we have that $m \in \{ 0, \pm 1\}$, and $x{\bf a} \in C({\bf a}, s_\beta w)$.  Now, if $s_\beta \in W_M$, then $R_N \subseteq s_\beta w(R^+)$ and $x{\bf a} \in \mathscr U$, a contradiction.  Thus $\beta \in \pm R_N$, and without loss of generality we may assume $\beta \in  R_N$.  Now in passing from $y{\bf a}$ to $x{\bf a}$, we crossed $H$ in the $\beta$-{\em opposite} direction, where by definition this means for any point $a$ in the interior of ${\bf a}$, $x(a) - y(a) \in \mathbb R_{<0}\beta^\vee$.  Indeed, if not then since $\beta \in w(R^+)$ the crossing $y{\bf a} |_H x{\bf a}$ is in the $w$-direction; in that case $x{\bf a}$ belongs to $C({\bf a},w)$ (since $y{\bf a}$ does), a contradiction.

To conclude, we observe that if ${\bf a} = {\bf a}_0, \dots, {\bf a}_\ell$ is a minimal gallery and crosses some $H_{\beta,m}$ with $\beta \in R_N$ in the $\beta$-opposite direction, then the terminal alcove ${\bf a}_\ell$ must actually lie outside of $\mathfrak P$ (since such a gallery must cross the hyperplane $H_{\beta,k(\beta,{\bf a})-1}$).
\end{proof}

\section{Reformulation of Theorem \ref{mainthm}}
\subsection{}
In the following reformulation of Theorem \ref{mainthm},
we assume $P = MN$ is semistandard and $x{\bf a}$ is a $P$-alcove. As in Beazley's
work \cite{Beazley}, it is easier to work with single cosets $xI$  than with
double cosets $IxI$, and  the next result allows us to do just that.

\begin{lemma}\label{mainthm_reformulated}  Theorem \ref{mainthm} is equivalent to the following statement: the map
$$
\phi: \,\, (^x I \cap I) \times^{\,^xI_M \cap I_M} xI_M \rightarrow xI
$$
given by $(i,m) \mapsto im\, \sigma (i)^{-1}$ is surjective.   Moreover, it is bijective if $x{\bf a}$ is a strict $P$-alcove.  In general, if $[i,xj]$ and $[i',xj']$ belong to the same fiber of $\phi$, then $xj$ and $xj'$ are $\sigma$-conjugate by an element of $^xI_M \cap I_M$.
\end{lemma}

\begin{proof}  It is straightforward to verify that the following diagram with vertical inclusion maps is Cartesian:
$$
\xymatrix{
(^xI \cap I) \times^{\, ^xI_M \cap I_M} x I_M \ar[d] \ar[r]^{\,\,\,\,\,\,\,\,\,\,\,\,\,\,\,\,\,\phi} & xI \ar[d] \\
I \times^{I_M} I_M x I_M \ar[r]^{\,\,\,\,\,\,\,\,\,\,\,\,\,\,\,\,\,\phi} & IxI.}
$$
The lemma is now clear by appealing to $I$-equivariance: each element of
$IxI$ is $\sigma$-conjugate under $I$ to an element of $xI$, and $\phi$ is
$I$-equivariant with respect to the action by $\sigma$-conjugation on $IxI$
and the action on $I \times^{I_M} I_M x I_M$ given by $i'[i,m] := [i'i,m]$
for $i' \in I$ and $[i,m] \in I \times^{I_M} I_M x I_M$.
\end{proof}

\smallskip

We can now prove the portion of Theorem \ref{mainthm} relating to the fibers
of $\phi$.  Suppose that $[i_1,xj_1], [i_2,xj_2] \in \,(^xI \cap I) \times^{\, ^xI_M \cap I_M} xI_M$ satisfy $i_1xj_1 \sigma(i_1)^{-1} = i_2xj_2 \sigma(i_2)^{-1}$.  Letting $i := i_2^{-1}i_1$, we see that
\begin{equation} \label{1}
x^{-1} i x = j_2 \sigma(i) j_1^{-1}.
\end{equation}
We have the Iwahori decompositions $I = I_{\overline{N}} I_M I_N$ and $^xI = \, ^xI_{\overline{N}} \, ^xI_M \, ^xI_N$, where $I_N := N \cap I$ and $I_{\overline{N}} := \overline{N} \cap I$.
Using our assumption that $x{\bf a}$ is a $P$-alcove, we deduce
\begin{equation} \label{intersection}
^xI \cap I = I_{\overline{N}} ~ (\, ^xI_M \cap I_M) \, ^xI_N.
\end{equation}
Write $i = i_- \, i_0 \, i_+$, with $i_- \in I_{\overline{N}}$, $i_0 \in \, ^xI_M \cap I_M$, and $i_+ \in \, ^xI_N$.  Using (\ref{1}) we get
\begin{equation} \label{2}
^{x^{-1}}i_- \cdot \, ^{x^{-1}}i_0 \cdot\, ^{x^{-1}}i_+ = \, ^{j_2}\sigma(i_-) \cdot j_2\sigma(i_0) j_1^{-1} \cdot \, ^{j_1}\sigma(i_+).
\end{equation}
By the uniqueness of the factorization of elements in $\overline{N} \cdot M \cdot N$, we get
\begin{align}
^{x^{-1}}i_- &= \, ^{j_2}\sigma(i_-) \label{inj_minus} \\
^{x^{-1}}i_0 &= j_2\sigma(i_0)j_1^{-1} \label{inj_0}\\
^{x^{-1}}i_+ &= \, ^{j_1}\sigma(i_+) \label{inj_plus}.
\end{align}

{}From (\ref{inj_0}), we deduce that $xj_1$ is $\sigma$-conjugate to $xj_2$ by an
element in $^xI_M \cap I_M$.  This proves the main assertion regarding the fibers of
$\phi$.

It remains to prove that $\phi$ is injective when $x{\bf a}$ is a strict $P$-alcove.  In that case conjugation by $x$ is strictly expanding (resp.~contracting) on $I_{\overline{N}}$ (resp.~on $I_N$).  In other words, the condition (\ref{condition(2)}) hence also (\ref{condition(2A)}) holds with the inclusions replaced by strict inclusions.  But then (\ref{inj_minus}) (resp. (\ref{inj_plus})) can hold only if $i_- = 1$ (resp. $i_+ = 1$).  Thus, in that case we have $i = i_0 \in \, ^xI_M \cap I_M$, and it follows that $[i_1, xj_1] = [i_2, xj_2]$.  This proves the desired injectivity of $\phi$. \qed

\section{A variant of Lang's theorem for vector groups} \label{Lang_variant_section}
\subsection{}
As before, let $k$ denote a finite field with $q$ elements, and let $\overline k$
denote an algebraic closure of $k$. We write $\sigma$ for the Frobenius
automorphism $x \mapsto x^q$ of $\overline k$. In this section we will be
concerned with an automorphism $\tau$ of $\overline k$, which is required to be
either $\sigma$ or $\sigma^{-1}$.
By a $\tau$-space $(V,\Phi)$ we mean a finite dimensional vector space $V$
over
$\overline k$ together with a $\tau$-linear map $\Phi:V \to V$. We do \emph{not}
require that $\Phi$ be bijective. The category of $\tau$-spaces is abelian
and every object in it has finite length.
Let  $(V,\Phi)$ be a simple object in this category. We claim that $V$ is
$1$-dimensional (cf.~the proof of Lemma 1.3 in  \cite{kr03}). Since $\ker
\Phi$ is a subobject of
$V$, we must have either $\ker \Phi=V$ or $\ker \Phi=0$. In the first case
$\Phi=0$, every subspace is a subobject, and therefore simplicity forces $V$
to be
$1$-dimensional. In the second case $\Phi$ is bijective, and a subspace $W$
is a subobject $\iff$ $\Phi W=W$ $\iff$ $\Phi^{-1}W=W$. Therefore we may as
well assume that $\tau=\sigma$ (since $\Phi^{-1}$ is $\sigma$-linear if
$\Phi$ is $\sigma^{-1}$-linear). Then by Lang's theorem for general linear
groups over $k$, our $\tau$-space is a direct sum of copies of $(\overline
k,\sigma)$, hence due to simplicity is $1$-dimensional.
\begin{lemma}
Let $(V,\Phi)$ be a $\tau$-space. Then the $k$-linear map $v \mapsto
v-\Phi(v)$ from $V$ to $V$ is surjective.
\end{lemma}
\begin{proof}
Filter $(V,\Phi)$ so that each successive quotient is $1$-dimensional. Since
the desired surjectivity follows from surjectivity of the induced map on the
associated graded object, we just need to prove surjectivity when $V$ is
$1$-dimensional. This amounts to the solvability of
the equations $x-ax^q=b$ and $x-ax^{1/q}=b$. Solvability of the first
equation is obvious, and so too is that of the second after the change
of variables
$x=y^q$, which leads to the equivalent equation
$y^q-ay=b$.
\end{proof}
\begin{cor} \label{Lang_variant_cor}
Let $V_0$ be a finite dimensional $k$-vector space, let $V=V_0\otimes_k \overline
k$, and let $M:V \to V$ be a linear map. Then
\begin{enumerate}
\item[(1)] for every $w \in V$ there exists $v \in V$ such that $\sigma v-Mv=w$,
and
\item[(2)] for every $w \in V$ there exists $v \in V$ such that $v-M\sigma v=w$.
\end{enumerate}
\end{cor}
\begin{proof}
The second statement follows from the lemma (with $\tau=\sigma$), and the
first follows from the lemma (with $\tau=\sigma^{-1}$) after making the
change of variables $v=\sigma^{-1} v'$.
\end{proof}
\begin{rem}
We note that the second statement of the corollary can also be proved in the
same way as Lang's theorem. However this method  does not
handle the first statement of the corollary in the case when $M$ is not
bijective. 
\end{rem}

\section{Proof of surjectivity in Theorem \ref{mainthm}}
\subsection{The method of successive approximations}
Again assume that $x{\bf a}$ is a $P$-alcove.
Recall that by Lemma \ref{mainthm_reformulated}, we need to prove the surjectivity of the map
$$
(^x I \cap I) \times xI_M \rightarrow xI
$$
given by $(i,m) \mapsto im\, \sigma (i)^{-1}$.  In other words, given an
element of $xI$, we can $\sigma$-conjugate it by an element of $^xI \cap I$
into the set $xI_M$.

Define the normal subgroup $I_n \subset I$, $n = 0,1,2, \dots $,
to be the $n$-th principal congruence subgroup of $I$.
More precisely, let $\mathcal G$ denote the
Bruhat-Tits parahoric $\mathfrak o$-group scheme corresponding to $I$, so
that $\mathcal G(\mathfrak o) = I$.  For $n \geq 0$, let $I_n$ denote the kernel of
$\mathcal G(\mathfrak o)
\twoheadrightarrow \mathcal G(\mathfrak o/\epsilon^n \mathfrak o)$.

Define the normal subgroups $N_n \subset N(\mathfrak o) \cap I$,
$\overline{N}_n \subset \overline{N}(\mathfrak o) \cap I$ and $M_n \subset
M(\mathfrak o)\cap I$ to be the intersections $I_n \cap N$ resp. $I_n \cap
\overline{N}$ resp. $I_n \cap M$.  For each $n \geq 0$, we have the Iwahori
factorization
$$
I_n = M_n N_n \overline{N}_n = \overline{N}_n N_n M_n.
$$
We have the relations
\begin{align} \label{condition(2'')}
^xN_n &\subseteq N_n \\
^x\overline{N}_n &\supseteq \overline{N}_n \notag
\end{align}
which follow from our assumption that $x{\bf a}$ is a $P$-alcove.

Conjugating by $x$ the decomposition $I = I_M I_N I_{\overline{N}}\,$ yields $^xI = \, ^xI_M \, ^xI_N \,
^xI_{\overline{N}}$.  By our assumptions on $x$, we have
$$
^xI \cap I = (\, ^xI_M \cap I_M) \, ^xI_N \, I_{\overline{N}}.
$$
Similarly, for each $n \geq 0$, we have
$$
^xI_n \cap I_n = (\, ^xM_n \cap M_n) \, ^xN_n \, \overline{N}_n.
$$

The next lemma is a key ingredient in the proof of Theorem \ref{mainthm}.
Here and in the remainder of this section we use the following notation:
for $h \in G(L)$, a superscript $^h-$ stands for conjugation by $h$, and a
superscript $^\sigma-$ means application of $\sigma$, so in
particular, for $g, h \in G(L)$, the symbol $^{h\sigma}g$ will stand for
$h\sigma(g)h^{-1}$, and $^{\sigma h}g$ will stand for
$\sigma(h)\sigma(g)\sigma(h^{-1})$.

\begin{lemma} \label{sided_approx_lemma}  Fix an element $m \in I_M$ and an integer $n \geq 0$.
\begin{enumerate}
\item[(i)] Given $i_- \in \overline{N}_n$, there exists $b_- \in \overline{N}_n$ such that
$^{(xm)^{-1}}b_- i_- \, ^{\sigma}b_-^{-1} \in \overline{N}_{n+1}$.
\item[(ii)] Given $i_+ \in N_n$, there exists $b_+ \in N_n$ such that
$b_+ i_+ \, ^{mx\sigma}b_+^{-1} \in N_{n+1}$.
\end{enumerate}
\end{lemma}

\begin{proof}
Borrowing the notation of \cite{GHKR}, $\S5.3$, the group $N$ possesses a finite separating
filtration by normal subgroups
$$
N = N[1] \supset N[2] \supset \cdots
$$
defined as follows.  Choose a Borel subgroup $B'$ containing $A$ and
contained in $P$; use $B'$ to determine a notion of (simple) positive root
for $A$ acting on ${\rm Lie}(G)$.  Let $\delta'_N$ be the cocharacter in
    $X_*(A/Z)$ (where $Z$ denotes the center of $G$) which is the sum of
    the $B'$-fundamental coweights $\varpi_\alpha$, where $\alpha$ ranges
    over the simple $B'$-positive roots for $A$ appearing in ${\rm
    Lie}(N)$.  Then let $N[i]$ be the product of the root groups $U_\beta
    \subset N$ for $\beta$ satisfying $\langle \beta, \delta'_N \rangle
    \geq i$.  The subgroups $N[i]$ are stable under conjugation by any
    element in $M$ (as one can check using the Bruhat decomposition of $M$ with respect to the Borel subgroup $B' \cap M$).  The successive quotients $N\langle i \rangle := N[i]/N[i+1]$ are abelian (see loc.~cit.).

We define $N_n[i] := N_n \cap N[i]$, and $N_n\langle i \rangle := N_n[i]/N_n[i+1]$.  We define the groups
$\overline{N}[i]$, $\,\overline{N}\langle i \rangle$, $\,\overline{N}_n[i]$, and $\overline{N}_n\langle i \rangle$ in an analogous manner.

Now we are ready to prove statement (i).  Note that the successive quotients $\overline{N}_n\langle i \rangle$ are abelian, and moreover $\overline{N}_{n+1}\langle i \rangle$ is a subgroup of $\overline{N}_n\langle i \rangle$, and the quotient
$$
\overline{N}_n\langle i \rangle/\overline{N}_{n+1}\langle i \rangle
$$
is a vector group over the residue field of $\mathfrak o$.  Conjugation by $m^{-1} \in I_M$ or $x^{-1}$ preserves
$\overline{N}_n$ as well as each
$\overline{N}_n[i]$ and $\overline{N}_n\langle i \rangle$ (for $x^{-1}$, we use (\ref{condition(2'')}) above).   Hence the map
$b_- \mapsto \, ^{(xm)^{-1}}b_- \, ^{\sigma}b_-^{-1}$ induces on each vector
group $\overline{N}_{n}\langle i \rangle/\overline{N}_{n+1}\langle i
\rangle$ a map like that considered in Corollary \ref{Lang_variant_cor} (1).  Using
that lemma repeatedly on these quotients in a suitable order, we may find
$b_- \in \overline{N}_n$ such that
$$
^{(xm)^{-1}}b_- i_- \, ^{\sigma}b_-^{-1} \in \overline{N}_{n+1},
$$
thus verifying part (i).

Now for part (ii) we use a very similar argument.  Conjugation by $mx$
preserves $N_n$ (for $x$ we use (\ref{condition(2'')}) above), as well as
each $N_n[i]$ and $N_n\langle i \rangle$.  Hence the map $b_+ \mapsto \,
b_+ \, ^{mx\sigma}b_+^{-1}$ induces on each vector group $N_n\langle i
\rangle/N_{n+1}\langle i \rangle$ a map like that considered in Corollary
\ref{Lang_variant_cor} (2).  We conclude as in part (i) above.  This
completes the proof of the lemma.

\end{proof}

Now we continue with the proof of Theorem \ref{mainthm}.
The Iwahori subgroup $I$ has the filtration $I \supset I_1 \supset I_2
\supset I_3 \supset \dots$ by principal congruence subgroups. We want to
refine this filtration to a filtration $I=I[0] \supset I[1] \supset I[2]
\supset I[3] \supset \dots$ satisfying the following conditions:
\begin{enumerate}
\item Each $I[r]$ is normal in $I$.
\item Each $I[r]$ is a semidirect product $I\langle r \rangle I[r+1]$,
where $I\langle r \rangle$ is either an affine root subgroup (hence
one-dimensional over our ground field $k$) or else contained in $A(\mathfrak o)$.
\end{enumerate}
One can construct such filtrations directly by inserting suitable terms into
the filtration by principal congruence subgroups.  It turns out to be much cleaner 
and more useful for other portions of this paper, to take instead a
generic Moy-Prasad filtration (see below for a discussion of these). In any case, we
fix one such filtration (which need not have any special properties relative
to our chosen $P=MN$).

We start with a $P$-alcove $x{\bf a}$ and an element $ y \in xI$. We want to find
an element  $ g \in {}^xI \cap I$  such that $gy{} \, \sigma(g)^{-1} \in
xI_M$. As usual we do this by
successive approximations, first $\sigma$-conjugating $y$ into $xI_M I[1]$,
then into $xI_M I[2]$, and so on. We have to take care that
the elements doing the  $\sigma$-conjugating approach $1$ as $r \to \infty$.  Assuming we can do this,
if $h^{(r)} \in \, ^xI \cap I$ is used to $\sigma$-conjugate the appropriate element of $xI_MI[r]$ into $xI_MI[r+1]$, then the convergent product
$$
g := \cdots h^{(2)}h^{(1)}h^{(0)}
$$
has the desired property.

So we need to show that any element  $ xi_Mi[r] \in xI_M I[r]$ is
$\sigma$-conjugate under ${}^xI \cap I$ to an element of $xI_M I[r+1]$ (and
that the
$\sigma$-conjugators can be taken to be small when $r$ is large). Use item
(2) to decompose $i[r]$ as $i\langle r \rangle i[r+1]$. There are two cases.
If $I\langle r \rangle \subset A(\mathfrak o)$, then we can absorb $i\langle
r \rangle $ into $i_M$, showing that our element already lies in $xI_M
I[r+1]$.

Otherwise $i\langle r \rangle$ lies in one of the affine root
subgroups of $I$; write $\alpha$ for the ordinary root obtained as the vector
part of our affine root.  If $\alpha$ is
a root in
$M$, then again we absorb $i\langle
r \rangle $ into $i_M$ and do not need to $\sigma$-conjugate.
Otherwise $\alpha$ is a root in $N$ or
$\overline N$, and in either case we may use the Lang
theorem variant (i.e. the appropriate statement in Lemma \ref{sided_approx_lemma}) to produce an element
$  h \in {}^xI \cap I$ (suitably
small when
$r$ is large) such that
$$hxi_Mi\langle r \rangle \sigma(h)^{-1} =xi_M i',
$$
for some $i' \in I[r+1]$.  (For example, if $i\langle r \rangle \in N_n$
take $h := \, ^xb_+$, where $b_+$ is the element produced in Lemma \ref{sided_approx_lemma} (ii) for
$m := i_M$ and $i_+ := m i\langle r \rangle m^{-1}$.)
Then
\[
hxi_Mi\langle r \rangle i[r+1] \sigma(h)^{-1} =xi_M i' \,(\sigma(h)
i[r+1]\sigma(h)^{-1})
\in xI_M I[r+1],
\]
as desired. (We  used here that $I[r+1]$ is normal in $I$.)   Lemma \ref{sided_approx_lemma} produces elements $h$ which are suitably small when $r$ is large, so that we are done, modulo the information on Moy-Prasad filtrations which follows.

\smallskip

\subsection{Moy-Prasad filtrations} Our reference for Moy-Prasad filtrations is \cite{MP}.
Recall that Moy-Prasad filtrations on
$I$  are obtained from points $x$ in the base alcove ${\bf a}$.  On the
Lie algebra this works as follows. The vector space $\mathfrak g \otimes_k
k[\epsilon,\epsilon^{-1}]$ is graded by the group $X^*(A) \oplus \mathbb Z$
(since $\mathfrak g$ is graded by $X^*(A)$ and $k[\epsilon,\epsilon^{-1}]$
is graded by $\mathbb Z$). (For the moment $k$ is any field.) The pair
$(x,1)$ gives a homomorphism
$X^*(A)\oplus \mathbb Z \to \mathbb R$, which we use to obtain an $\mathbb
R$-grading on $\mathfrak g \otimes_k
k[\epsilon,\epsilon^{-1}]$, as well as an associated $\mathbb R$-filtration.
We also obtain an $\mathbb R$-filtration on the completion $\mathfrak g(F)$
of
$\mathfrak g
\otimes_k k[\epsilon,\epsilon^{-1}]$. Thus, for $r \in \mathbb R$ the
subspace  $\mathfrak g(F)_{\ge r}$ is the completion of the direct sum of
the affine weight spaces of weight (with respect to $(x,1)$) greater than or
equal to $r$, which for the affine weight space $\epsilon^n \mathfrak a$
means that $n \ge r$, and for an affine weight space $\epsilon^n \mathfrak
g_\alpha$ ($\alpha$ being an ordinary root) means that $\alpha(x)+n \ge r$.
Of course $\mathfrak g(F)_{\ge 0}$ is the Iwahori subalgebra obtained as the
Lie algebra of $I$ \footnote{Warning: This description is incompatible with
the normalization of the correspondence between alcoves and Iwahori subgroups we are
using in this paper: it turns out $G(F)_{\ge 0}$ is really ``opposite'' to our Iwahori $I$.  To get our $I$, we
should instead define $\mathfrak g(F)_{\ge r}$ to be the completion of the sum of the affine weight spaces of weight (with respect to $(x,-1)$)
less than or equal to $-r$.}.  It is clear that $[\mathfrak g(F)_{\ge r},\mathfrak
g(F)_{\ge s}] \subset \mathfrak g(F)_{\ge r+s}$, from which it follows that
$\mathfrak g(F)_{\ge r}$ is an ideal in $\mathfrak g(F)_{\ge 0}$ whenever
$r$ is non-negative.

When $r$ is non-negative, the Moy-Prasad subgroups $G(F)_{\ge
r}$ of $G(F)$ are by definition the subgroups generated by suitable subgroups of
$A(\mathfrak o)$ and of the various root subgroups, in such a way that the
Lie algebra of
$G(F)_{\ge r}$ ends up being $\mathfrak g(F)_{\ge r}$. In characteristic $0$
the fact that $\mathfrak g(F)_{\ge r}$ is an ideal in $\mathfrak g(F)_{\ge
0}$ implies that $G(F)_{\ge r}$ is normal in $I=G(F)_{\ge 0}$.
Moy and Prasad prove normality in the general case from other considerations.
In our present situation, where $G$ is split, it is straightforward
to prove the normality using commutator relations for
the various affine root groups $U_{\alpha +n}$ in $G(F)$.

What does it mean for $x$ to be a \emph{generic} element in the base alcove?
For an arbitrary point $x$ in the standard apartment it may accidentally
happen that the homomorphism $(x,1):X^*(A)\oplus \mathbb Z \to \mathbb R$
sends two distinct affine weights occurring in $\mathfrak g
\otimes_k k[\epsilon,\epsilon^{-1}]$ to the same
real number. When such an accident never occurs, we say that $x$ is generic.
The set of non-generic points in the standard apartment is a locally finite
union of affine hyperplanes, including all the affine root hyperplanes, but
also those obtained by setting any difference of roots equal to an integer.
In the case of ${\rm SL}(2)$, all points in the base alcove but its midpoint are
generic. In general one can at least say that the set of generic points in
the base alcove is non-empty and open. When $x$ is generic, then going down
the Moy-Prasad filtration strips away affine weight spaces, one-by-one, just
as we want.

\smallskip

\subsection{A refinement}

It is clear that in case $^xI_M = I_M$, we can do better: we can
$\sigma$-conjugate any element in $xI_M$ to $x$ using an element of $I_M$.  To see this we adapt
 the proof of Lang's theorem to prove the surjectivity of the map
$I_M \rightarrow I_M$ given by $h \mapsto h^{-1}\,\, ^{x\sigma }h$.
Indeed, $I_M$ has a filtration by normal subgroups which are stabilized by
${\rm Ad}(x)$, such that our map induces on the successive quotients a
finite \'etale surjective map (take the Moy-Prasad filtration on $I_M$
corresponding to the barycenter of the alcove in the reduced building for
$M(L)$ corresponding to $I_M$).  Using the surjectivity just proved, given
$i \in I_M$ we find an $h \in I_M$ solving the equation
$xix^{-1} = h^{-1}\, \, ^{x\sigma}h$.  We then have $h(xi)\sigma(h)^{-1} = x$.  Thus,
we have proved the following proposition.

\begin{prop} \label{sigma_conj_prop}
Suppose $x \in \widetilde{W}_M$ is such that there exists a semistandard parabolic
subgroup $P=MN$ having the property that $^xI_N \subseteq I_N$, i.~e.~such
that $x\mathbf a$ is a $P$-alcove.  Then any element
of $xI$ is $\sigma$-conjugate to an element of $xI_M$ using an element of
$^xI \cap I$.  If moreover, $^xI_M = I_M$, then we may $\sigma$-conjugate
any element of $xI$ to $x$, using an element of $^xI \cap I$.
\end{prop}

Given an element $x \in \widetilde{W}_M$ such that $^xI_M
= I_M$, in general there is no parabolic $P = MN$ such that $^xI_N
\subseteq I_N$ and $^{x^{-1}}I_{\overline{N}} \subseteq I_{\overline{N}}$
(see also the discussion after Definition \ref{defn_std_repr} below). However, when
$M$ is adapted to $I$ in the sense of Definition \ref{def.adap}, such $P$ does exist,
as is shown in Proposition \ref{prop.adap}.

\section{Review of $\sigma$-conjugacy classes}\label{sec.revbg}
\subsection{Classification of $\sigma$-conjugacy classes}
We recall the
description of the set $B(G)$ of $\sigma$-conjugacy classes in $G(L)$; for
details see \cite{kottwitz85}, \cite{Ko-isoII} 5.1, and \cite{Kottwitz-defect} 1.3.  We denote by
$\Lambda_G$ the quotient of $X_*(A)$ by the coroot lattice; this is the
algebraic fundamental group of $G$. We can identify $\Lambda_G$ with the
group of connected components of the loop group $G(L)$. Let $\eta_G \colon
G(L) \longrightarrow \Lambda_G$ be the natural surjective homomorphism, as
constructed in \cite{Ko-isoII}, $\S 7$ and denoted there by $\omega_G$; it
is sometimes called the \emph{Kottwitz homomorphism}.
Analogously, we denote
by $\Lambda_M$ the quotient of $X_*(A)$ by the coroot lattice for $M$, and
by $\eta_M$ the corresponding homomorphism.

If $P=MN$ is a standard parabolic subgroup of $G$ with unipotent radical
$N$ and $M$ the unique Levi containing $A$, then the set $\Delta$ of simple
roots for $G$ decomposes as the disjoint union of $\Delta_M$ and
$\Delta_N$, where $\Delta_M$ is the set of simple roots of $M$, and
$\Delta_N$ is the set of those simple roots for $G$ which occur in the Lie
algebra of $N$. We write $A_P$ (or $A_M$) for the connected component of the center of $M$, and we let $\mathfrak a_P$ denote the real vector space $X_*(A_P) \otimes \mathbb R$.  As usual, $P$ determines an open chamber $\mathfrak a^+_P$ in $\mathfrak a_P$ defined by
$$
\mathfrak a_P^+ = \{ v \in \mathfrak a_P ~ : ~ \langle \alpha, v \rangle > 0, \,\mbox{for all $\alpha \in \Delta_N$} \}.
$$
The composition $X_*(A_P) \hookrightarrow X_*(A) \twoheadrightarrow \Lambda_M$, when tensored with $\mathbb R$,
yields a canonical isomorphism $\mathfrak a_P \cong \Lambda_M \otimes \mathbb R$.  Let $\Lambda_M^+$ denote the subset of elements in $\Lambda_M$ whose image under $\Lambda_M \otimes \mathbb R \cong \mathfrak a_P$ lies in $\mathfrak a^+_P$.

Let $\mathbb D$ be the diagonalizable group over $F$ with character group $\mathbb Q$.
As in \cite{kottwitz85}, an element $b \in G(L)$ determines a homomorphism
$\nu_b: \mathbb D \rightarrow G$ over $L$, whose $G(L)$-conjugacy class
depends only on the $\sigma$-conjugacy class $[b] \in B(G)$.  We can assume
this homomorphism factors through our torus $A$, and that the corresponding
element $\overline{\nu}_b \in X_*(A)_\mathbb Q$ is dominant.  Then $b
\mapsto \overline{\nu}_b$ is called the {\em Newton map} (relative to the
group $G$).  Recall that $b \in G(L)$ is called {\em basic} if $\nu_b$ factors through the center $Z(G)$ of $G$.

We shall use some properties of the Newton map.   We can identify the quotient $X_*(A)_{\mathbb Q}/W$ with the closed
dominant chamber $X_*(A)_{\mathbb Q}^+$.
 The map
\begin{align} \label{Newton_point_Kottwitz_point}
B(G) &\rightarrow X_*(A)^+_\mathbb Q \times \Lambda_G \\
b &\mapsto (\overline{\nu}_b, \eta_G(b)) \notag
\end{align}
is injective (\cite{Ko-isoII}, 4.13).

The Newton map is functorial, such that we have a commutative
diagram
\begin{equation} \label{Newton_is_functorial_diagram}
\xymatrix{ B(M) \ar[r] \ar[d] & B(G)\ar[d] \\
X_*(A)_{\mathbb Q}/W_M  \times \Lambda_G \ar[r] & X_*(A)_{\mathbb Q}/W  \times  \Lambda_G
}
\end{equation}
and moreover the vertical arrows, given by ``(Newton point, Kottwitz point)'', are {\em injections}.  Indeed, the right vertical arrow is the injection (\ref{Newton_point_Kottwitz_point}).  To show the left vertical arrow is injective, it is enough to prove that if $b_1,b_2 \in M(L)$ have the same Newton point and the same image under $\eta_G$, then they have the same image under $\eta_M$.  We may assume that $b_1, b_2 \in \widetilde{W}_M$ (see Corollary \ref{tildeW_to_B(G)} below); for $i = 1,2$ write $b_i = \epsilon^{\lambda_i}w_i$ for $\lambda_i \in X_*(A)$ and $w_i \in W_M$.  Let $Q^\vee $ (resp. $Q^\vee_M$) denote the lattice generated by the coroots of $G$ (resp. $M$) in $X_*(A)$.  The equality $\eta_G(b_1) = \eta_G(b_2)$ means that $\lambda_1 - \lambda_2 \in Q^\vee$.  The equality $\overline{\nu}_{b_1} =
\overline{\nu}_{b_2}$ implies that $\lambda_1 - \lambda_2 \in Q^\vee_M \otimes \mathbb R$.  It follows that $\lambda_1 - \lambda_2 \in Q^\vee_M$, and this is what we wanted to prove.

The following lemma is a direct consequence of the commutativity of the
diagram above.

\begin{lemma}\label{BMvsBG}
Let $M\subset G$ be a Levi subgroup containing $A$. If $[b']_M \subset [b]$
for some $b' \in M(L)$, then $\overline{\nu}_b = \overline{\nu}_{b',
    G-{\rm dom}}$ as elements of $X_*(A)_{\mathbb Q}^+$.
\end{lemma}
Here $\overline{\nu}_{b'}$ is the Newton point of $b'$ (viewed as an element of $M(L)$) and $\overline{\nu}_{b', G-{\rm dom}}$ denotes the unique $G$-dominant element of $X_*(A)_\mathbb Q$ in its $W$-orbit.

\medskip

We denote by $\lambda_M$ the canonical map
\begin{equation} \label{Newton_map_for_basic}
\lambda_M\colon \Lambda_M = X^*(Z(\widehat{M})) \rightarrow
X^*(Z(\widehat{M}))_{\mathbb R} = X_*(Z(M))_{\mathbb R} \hookrightarrow X_*(A)_{\mathbb R}.
\end{equation}
This can be identified with the map
\begin{equation*}
\Lambda_M \rightarrow X_*(A_M)_\mathbb Q \hookrightarrow X_*(A)_\mathbb Q
\end{equation*}
where the first arrow is given by averaging the $W_M$-action.  Next we define the following subsets of $X_*(A)^+_\mathbb Q$: the subset $\mathcal N_G$ consists of all Newton points $\overline{\nu}_b$ for $b \in B(G)$, and $\mathcal N^+_M$ consists of the images of elements of $\Lambda_M^+$, under the map $\lambda_M$.  We have the equality
\begin{equation} \label{Newton_point_decomp}
\mathcal N_G = \coprod_{P =  MN} \mathcal N^+_{M},
\end{equation}
the union ranging over all standard parabolic subgroups of $G$.

This equality results from two facts.  First, we are taking the Newton points associated to elements of $B(G)$ and making use of the decomposition of $B(G)$
$$
B(G) = \coprod_P B(G)_P,
$$
where $P$ ranges over standard parabolic subgroups and $B(G)_P$ is the set of elements $[b] \in B(G)$ such that $\overline{\nu}_b \in \mathfrak a^+_P$ (see \cite{kottwitz85, Ko-isoII}); note that elements in $B(G)_P$ can be represented by basic elements in $M(L)$ (\cite{Ko-isoII}, 5.1.2).  Second, for $b$ a basic element in $M(L)$ (representing e.~g.~an element in $B(G)_P$) its Newton point $\overline{\nu}_b$ is the image of $\eta_M(b) \in \Lambda_M$ under $\lambda_M$.
This follows from the characterization of $\overline{\nu}_b$ in \cite{kottwitz85}, 4.3 (applied to $M$ in place of $G$), together with (\ref{Newton_is_functorial_diagram}).

\begin{rem}
The right hand side in (\ref{Newton_point_decomp}) is easy to enumerate for any given group (with the aid of a computer).  This fact makes feasible our computer-aided verifications of our conjectures relating to the non-emptiness of $X_x(b)$, see section \ref{consequences}.  Moreover, the injectivity of (\ref{Newton_point_Kottwitz_point}) together with (\ref{Newton_point_decomp}) gives a concrete way to check whether two elements in $G(L)$ are $\sigma$-conjugate.
\end{rem}

\subsection{Construction of standard representatives for $B(G)$}

Here we will define the \emph{standard representatives} of $\sigma$-conjugacy
classes in the extended affine Weyl group. First note that the map
$G(L)\rightarrow B(G)$ induces a map $\widetilde{W} \rightarrow B(G)$.  Our goal is to find special elements in $\widetilde{W}$ which parametrize the elements of $B(G)$.

Denote by $\Omega_G \subset \widetilde{W}$ the subgroup of elements of length
$0$.  Let $G(L)_{\rm b}$ resp. $B(G)_{\rm b}$ denote the set of basic elements resp.~basic $\sigma$-conjugacy classes in $G(L)$.  In the following lemma we recollect some standard facts relating the Newton map to the homomorphism $\eta_G : G(L) \twoheadrightarrow \Lambda_G$.  The connection between the two stems from fact that if $b \in G(L)$ is basic, then the Newton point $\overline{\nu}_b \in X_*(Z(G))_{\mathbb R}$ is the image of $\eta_G(b) \in \Lambda_G$ under the canonical map $\lambda_G\colon \Lambda_G \rightarrow X_*(A)_{\mathbb R}$ (see (\ref{Newton_map_for_basic})).

\begin{lemma} \label{construct_standard_reps}
\begin{enumerate}
\item[(i)]  The map $\eta_G$ induces a bijection $B(G)_{\rm b} ~ \widetilde{\rightarrow} ~ \Lambda_G$.
\item[(ii)] Elements in $\Omega_G \subset G(L)$ are basic, and the map $\eta_G$ induces a bijection $\Omega_G ~ \widetilde{\rightarrow} ~ \Lambda_G$.
\item[(iii)]  The canonical map $\Omega_G \rightarrow B(G)_{\rm b}$ is a bijection.
\end{enumerate}
\end{lemma}

\begin{proof}
First suppose $b \in \Omega_G$.
For sufficiently divisible $N > 1$, the element $b^N$ is a translation element
which preserves the base alcove, hence belongs to $X_*(Z(G))$.  The characterization of $\nu_b$ in \cite{kottwitz85}, 4.3, then shows that $b$ is basic, proving the first statement in (ii).  For part (i), recall that an isomorphism is constructed in loc.~cit.~5.6, and this is shown to be induced by $\eta_G$
in \cite{Ko-isoII}, 7.5.
Since $\eta_G$ is trivial on $I$ and $W_{\rm aff} \subset G_{\rm sc}(L)$, (i) and the Bruhat-Tits decomposition
$$
G(L) = \coprod_{w\tau \in W_{\rm aff} \rtimes \Omega_G} Iw\tau I
$$
imply that the composition
$$
\xymatrix{
\Omega_G \ar[r] & G(L)_{\rm b} \ar[r]^{\eta_G} &  \Lambda_G
}$$
is surjective.  Since this composition is easily seen to be injective, (ii) holds.  Part (iii) follows using (i-ii).
\end{proof}

Here is a slightly different point of view of the lemma:
The basic conjugacy classes are in bijection with $\Lambda_G$, the group
of connected components of the ind-scheme $G(L)$ (or the affine flag
variety), and the bijection is given by just mapping each basic
$\sigma$-conjugacy class to the connected component it lies in.
The key point here is that the Kottwitz homomorphism agrees with the
natural map $G(L) \rightarrow \pi_0(G(L)) = \Lambda_G$; see
\cite{kottwitz85}, \cite{Pappas-Rapoport} \S 5.

As a consequence of the lemma (applied to $G$ and its standard Levi
subgroups), we have the following corollary.

\begin{cor} \label{tildeW_to_B(G)}
The map $\widetilde{W}\rightarrow B(G)$ is
surjective.
\end{cor}

\begin{defn} \label{defn_std_repr}
For $[b] \in B(G)_P \subset B(G)$, we call the representative in $\,\Omega_{M} \subseteq
\widetilde{W}$ which we get from Lemma \ref{construct_standard_reps} (iii) the
\emph{standard representative} of $[b]$. Here standard refers back to our particular
 choice $B$ of Borel subgroup. If we made a different choice of Borel subgroup
containing $A$, we would get a different standard representative; all such
representatives will be referred to as \emph{semistandard}.
\end{defn}

The standard representative $b = \epsilon^\nu v$ hence satisfies
\begin{enumerate}
\item $b \in \widetilde{W}_{M}$, i.~e.~$v \in W_M$,
\item $b I_{M} b^{-1} = I_{M}$.
\end{enumerate}

\begin{rem}\label{rem.minu}
Let $x \in \Omega_G$ and write $x=\epsilon^\lambda w$ with $\lambda \in
X_*(A)$ and $w \in W$; we call $\lambda$ the {\em translation part} of $x$.
Then $\lambda$ is the (unique) dominant minuscule
coweight whose image in $\Lambda_G$ coincides with that of $x$. Indeed,
 since $x$ preserves the
base alcove $\mathbf a$,  the transform of the origin by $x$,
namely
$\lambda$, lies in the closure of the base alcove. This is what it means to
be dominant and minuscule.

Now consider \emph{standard} (semistandard is not enough) $P=MN$ and $x\in
\Omega_M$. Write
$x=\epsilon^\lambda w_M$ with $\lambda \in X_*(A)$ and $w_M \in W_M$. We
know that $\lambda$ is $M$-dominant and $M$-minuscule. We claim that
$x\mathbf a$ is a $P$-alcove  if and only if $\lambda$ is
dominant. Indeed,
$x\mathbf a$ is a  $P$-alcove if and only if $xI_Nx^{-1} \subset I_N$. Now
$w_MI_Nw_M^{-1}=I_N$, because $P$ was assumed standard. So $x\mathbf a$ is a
 $P$-alcove if and only if $\epsilon^\lambda I_N \epsilon^{-\lambda} \subset
I_N$ if and only if $\alpha(\lambda) \ge 0$ for all $\alpha \in R_N$ if and only if
$\alpha(\lambda)
\ge 0$ for all $\alpha >0$.
\end{rem}

\begin{example} \label{GL(n)_std_reps}
Let $G= GL_n$, let $A$ be the diagonal torus, and let $B$ be the Borel
group of upper triangular matrices. In this case, the Newton map is
injective. See \cite{Kottwitz-defect}, in particular the last paragraph of
section 1.3. We can view the Newton vector $\nu$ of a $\sigma$-conjugacy
class $[b]$ as a descending sequence $a_1 \ge \cdots \ge a_n$ of rational
numbers, satisfying an integrality condition. The standard parabolic
subgroup $P=MN$ is given by the partition $n = n_1 + \cdots + n_r$ of $n$
such that the $a_i$ in each corresponding batch are equal to each other,
and such that the $a_i$ in different batches are different. The standard
representative is (represented by) the block diagonal matrix with $r$
blocks, one for each batch of entries, where the $i$-th block is
\[
\left(
\begin{array}{cc}
0 &  \epsilon^{k_i+1} I_{k_i'} \\
\epsilon^{k_i} I_{n_i-k_i'} & 0
\end{array}
\right) \in GL_{n_i}(F).
\]
Here we write the entry $a_{n_1+\cdots+n_{i-1}+1} = \cdots =
a_{n_1+\cdots+n_i}$ of the $i$-th batch as $k_i + \frac{k_i'}{n_i}$ with
$k_i, k_i' \in \mathbb Z$, $0 \le k_i' < n_i$, which is possible by the
integrality condition, and $I_\ell$ denotes the $\ell\times\ell$ unit
matrix. It follows from the definitions that $k_i \ge k_{i+1}$ for all
$i=1,\dots, r-1$.
We see that the standard representative $x$ of $[b]$ has dominant
translation part if and only if for all $i$ with $k'_{i+1}\ne 0$ we have
$k_i > k_{i+1}$. Furthermore, this is equivalent to $x\mathbf a$ being a
$P$-alcove. If these conditions are satisfied, then $x\mathbf a$ is a
\emph{fundamental $P$-alcove} in the sense of Definition
\ref{fundamental-P-alcove}. 
\end{example}

\section{Proofs of Corollary \ref{sigma_quotient}(b) and Theorem \ref{HN}} \label{some_proofs}
\subsection{}
Assume $P = MN$ is semistandard and $x{\bf a}$ is a $P$-alcove.
There is a commutative diagram
\begin{equation} \label{diagram}
\xymatrix{
I_M x I_M /_\sigma \, I_M \ar[r]^{\sim} \ar[d]_{\cong} & IxI/_\sigma \, I \ar[d]_{\cong}  \\
\underset{[b'] \in B(M)_x}{\coprod} J^M_{b'}\backslash X^M_x(b')  \ar[r] &
\underset{[b] \in B(G)_x}{\coprod} J^G_b\backslash X^G_x(b).  }
\end{equation}
Here, for $[b'] \in B(M)_x$ we choose once and for all a representative $b' \in M(L)$;
for $[b] \in B(G)_x$  we also choose once and for all a representative $b \in G(L)$.
If under $B(M)_x
\rightarrow B(G)_x$, $[b'] \mapsto [b]$, then choose once and for all $c \in G(L)$
such that $c^{-1}b\sigma(c) = b'$.  In that case our choices yield the map
\begin{align*}
J^M_{b'} \backslash X^M_{x}(b') &\rightarrow J^G_b\backslash X^G_x(b) \\
m &\mapsto cm.
\end{align*}
We have now defined the bottom horizontal arrow.

Next we define the right vertical arrow.  Let an element of $IxI/_\sigma \, I$
be represented by $y \in IxI$.  There is a unique $[b] \in B(G)_x$ such that $y \in
[b]$.  Write $y = g^{-1}b\sigma(g)$ for some $g \in G(L)$.  Then the right vertical
map associates to $[y] = [g^{-1}b\sigma(g)]$ the $J^G_b$-orbit of $gI \in X^G_x(b)$.
The left vertical arrow is defined similarly.  It is easy to check that both vertical
arrows are bijective.  It is also clear that the diagram commutes.  The bijectivity
of the top horizontal arrow (Corollary \ref{sigma_quotient}(a)) thus implies the
surjectivity of the map $B(M)_x \rightarrow B(G)_x$ (in Corollary
\ref{sigma_quotient}(b)).

We now prove that $B(M)_x \rightarrow B(G)_x$ is also injective.
Given $b \in M(L)$, regard its Newton point $\overline{\nu}^M_b$ as an element in
$X_*(A)^+_\mathbb Q$, which denotes here the set of $M$-dominant elements of
$X_*(A)_\mathbb Q$.  The map
\begin{align*}
B(M) &\rightarrow X_*(A)^+_\mathbb Q \times \Lambda_M \\
b &\mapsto (\overline{\nu}^M_b, \eta_M(b))
\end{align*}
is injective, see (\ref{Newton_point_Kottwitz_point}).
Now suppose $b_1,  b_2 \in B(M)_x$ have the same image in $B(G)_x$.  Since
$\eta_M(b_1) = \eta_M(x) = \eta_M(b_2)$, by the preceding remark it is enough to show
that $\overline{\nu}_{b_1}^M = \overline{\nu}_{b_2}^M$.  We claim that our assumption
on $x$ forces each $\overline{\nu}_{b_i}^M$ to be not only $M$-dominant, but
$G$-dominant.  Indeed, $b_i$ is $\sigma$-conjugate in $M(L)$ to an element in $I_M x
I_M$, and since $^x(N \cap I) \subseteq N \cap I$, it follows that the isocrystal
$$({\rm Lie}\, N(L), {\rm Ad}(b_i) \circ \sigma)$$
comes from a crystal (i.e., there is some $\mathfrak o$-lattice
in ${\rm Lie} \, N(L)$ carried into itself by the $\sigma$-linear map ${\rm Ad}(b_i)
\circ
\sigma$; in fact, when $b_i$ itself lies in $I_MxI_M$, the lattice $\Lie N(L)\cap I$
does the job).  The slopes of any crystal are non-negative, which means in this
situation that
$\langle
\alpha,
\overline{\nu}_{b_i}^M \rangle
\geq 0$ for all
$\alpha \in R_N$.  This proves our claim.  Now since $\overline{\nu}_{b_1}^M$ and
$\overline{\nu}_{b_2}^M$ are conjugate under $W$ (cf.
(\ref{Newton_is_functorial_diagram})) they are in fact equal.  This completes the
proof of Corollary \ref{sigma_quotient}(b).

In light of the diagram (\ref{diagram}),
Theorem \ref{HN} follows from Corollary \ref{sigma_quotient}.  \qed

\section{Consequences for affine Deligne-Lusztig varieties} \label{consequences}
\subsection{}
In this section we present various consequences of Theorem \ref{HN},
and also some conjectures, relating to the non-emptiness and dimension of
$X^G_x(b)$.  We prove some parts of our conjectures.  Our conjectures have been
corroborated by ample computer evidence.  The computer calculations were done using
the ``generalized superset method'', that is, the algorithm implicit in Theorem
\ref{dim_alg}.  This will be discussed in section \ref{superset_section}.

\subsection{Translation elements $x = \epsilon^\lambda$}

Let us examine the non-emptiness of $X_x(b)$ in a very special case.

\begin{cor}
Suppose $x = \epsilon^\lambda$.  Then $X_x(b) \neq \emptyset$ if and only
if $[b] = [\epsilon^\lambda]$ in $B(G)$.
\end{cor}

\begin{proof}
There is a choice of Borel $B' = AU'$ such that $x{\bf a}$ is a $B'$-alcove
($\lambda$ is $B'$-dominant for an appropriate choice of $B'$).  Thus, by Theorem
\ref{HN} with $M = A$, we see $X^G_x(b) \neq \emptyset$ if and only if $b$ is
$\sigma$-conjugate to a translation $\epsilon^\nu$ for $\nu \in X_*(A)$, and
$X^A_x(\epsilon^\nu) \neq \emptyset$.  But the latter inequality holds if and only if
$\lambda = \nu$.
\end{proof}

\begin{rem} As G. Lusztig pointed out, the Corollary has a simple direct proof
in the special case where $G$ is simply-connected and $b=1$.
Let $x = \epsilon^\lambda$ and suppose $\lambda$ belongs to the coroot lattice.
Suppose $g^{-1}\sigma(g) \in IxI$.  Since the affine flag variety is of ind-finite
type, the Iwahori subgroup  $^gI$ is fixed by $\sigma^r$ for some $r >0$.  Thus,
$g^{-1}\sigma^r(g) \in I$.   On the other hand, $g^{-1}\sigma^r(g) \in IxI \cdots
IxI$ (product of $r$ copies of $IxI$), which since the lengths add is just
$I\epsilon^{r\lambda}I$.  This intersects $I$ only if $\lambda = 0$.
\end{rem}

\subsection{A necessary condition for the non-emptiness of $X_x(b)$}

We want to use Theorem~\ref{HN} to obtain results about affine Deligne-Lusztig varieties. Clearly, whenever $X_x(b)\ne\emptyset$, then $x$ and $b$ must lie in the same connected component of the loop group, i.e.~$\eta_G(x)=\eta_G(b)$. Whenever we can use Theorem~\ref{HN} to relate $X_x(b)$ to an affine Deligne-Lusztig variety for a Levi subgroup $M$, then we will get a similar necessary condition with respect to $\eta_M$. Typically, $\Lambda_M$ is much larger than $\Lambda_G$, so the condition for $M$ will be a much stronger restriction.

However, one has to be careful here, because the intersection of $M(L)$ with the $G$-$\sigma$-conjugacy class $[b]$ will in general consist of several $M$-$\sigma$-conjugacy classes. Here is what we can say:

\begin{prop} \label{nec_cond_prop}
Fix a $\sigma$-conjugacy class $[b]$ in $G$ with Newton vector
$\overline{\nu}_b$, and an element $x \in \widetilde{W}$.   If $X^G_x(b)
\neq \emptyset$, then the following holds: if $P=MN$ is a semistandard parabolic
subgroup such that $x{\bf a}$ is a $P$-alcove, then $\eta_G(x) = \eta_G(b)$ and
\begin{equation} \label{necessary_condition}
\eta_M(x) \in \eta_M(W\overline{\nu}_b \cap \mathcal N_M),
\end{equation}
where $\mathcal N_M$ denotes the image of $B(M)$ in $X_*(A)_{\mathbb Q}^{M-{\rm dom}}$
under the Newton map.
\end{prop}

The set $W\overline{\nu}_b \cap \mathcal N_M$ is the finite set of $M$-dominant elements of $X_*(A)_{\mathbb Q}$ that are $W$-conjugate to $\overline{\nu}_b$ and arise as the Newton point of some element of $M(L)$. See Example~\ref{example.illustrate931} below for a specific example. If $b$ is basic, then the statement of Proposition \ref{nec_cond_prop} simplifies.  We will consider the basic case in the next subsection.

Our condition (\ref{necessary_condition}) means that $x$ has the same value under $\eta_M$ as an element $b' \in M(L)$ with $\overline{\nu}_{b'}^M \in W\overline{\nu}_b$.  By the injectivity of the left vertical arrow of (\ref{Newton_is_functorial_diagram}), for a fixed $[b]$ there are only {\em finitely many} $\sigma$-conjugacy classes $[b'] \in B(M)$ such that $\overline{\nu}^M_{b'} \in W\overline{\nu}_b$ and $\eta_G(b') = \eta_G(b)$.  In particular, the condition that $\eta_M(x) = \eta_M(b')$ for some such $b' $ is a condition which we can check with a computer.

\begin{proof}
Condition (\ref{necessary_condition}) is a direct consequence of Theorem \ref{HN}.  Indeed, we know from part (a) of that theorem that $[b] = [b']$ for some $b' \in M(L)$, and that $X^M_{x}(b') \neq \emptyset$, which implies in turn that $\eta_M(x) = \eta_M(b')$.  Lemma \ref{BMvsBG} then shows that $\overline{\nu}^M_{b'} \in W\overline{\nu}_b$, as desired.
\end{proof}

\begin{example}\label{example.illustrate931}
Let $G=SL_3$, $P_2 = B \cup Bs_2 B$, and $P= {^{s_1}}P_2$. As in the proposition, write $P=MN$. In terms of matrices, we have
\[
M = \left(
\begin{array}{ccc}
* &   & * \\
  & * &   \\
* &   & *
\end{array}
\right),\qquad
N = \left(
\begin{array}{ccc}
1 &   &   \\
* & 1 & *  \\
  &   & 1
\end{array}
\right), \qquad
I\cap N = \left(
\begin{array}{ccc}
1 &   &   \\
\mathfrak o & 1 & \epsilon\mathfrak o  \\
  &   & 1
\end{array}
\right).
\]
Assume that the Newton vector of $b$ is $\overline{\nu}_b=(1, -\frac 12, -\frac 12)$. We have $W\overline{\nu}_b \cap \mathcal N_M = \{ (-\frac 12, 1, -\frac 12) \}$.

Now consider an element $x = \epsilon^\mu s_1 s_2 s_1 \in \widetilde{W}_M$, $\mu = (\mu_1, \mu_2, \mu_3)$, and assume that $x$ is a $P$-alcove, i.e., $\mu_2 -\mu_1 \ge -1$ and $\mu_2 -\mu_3\ge 1$. The proposition states that $X_x(b)=\emptyset$ unless $(\mu_1+\mu_3, \mu_2) = \eta_M(x) = (-1, 1)$. This is equivalent to $\mu_2 = 1$ since $\sum\mu_i = 0$, $x$ being an element of $SL_3$. Altogether we find that $X_x(b)=\emptyset$ unless $\mu$ is one of the four cocharacters $(-1,1,0)$, $(0,1,-1)$, $(1,1,-2)$, $(2,1,-3)$.
\end{example}

Note that Proposition \ref{nec_cond_prop} implies that for fixed $b$ and {\em proper} parabolic subgroup $P$, there are only finitely many $x$ such that $x{\bf a}$ is a $P$-alcove and for which $X_x(b)$ can be non-empty.

Proposition \ref{nec_cond_prop} provides an obstruction to the non-emptiness of affine
Deligne-Lusztig varieties: (\ref{necessary_condition}) must hold whenever $x{\bf a}$ is a $P$-alcove.
In the case where $[b]$ is basic, it seems
reasonable to expect that this is the only obstruction; see Conjecture
\ref{conj2} below. In the general case, it is clear that there are
additional obstructions. If $b$ is a translation element, then from Theorem
6.3.1 in \cite{GHKR} we see that whenever $X_x(b) \ne \emptyset$, there
exists $w\in W$ such that $x \ge {^w}b$ in the Bruhat order. (For general $b$, one can obtain a
similar criterion by passing to a totally ramified extension of $L$ where $b$
splits.) This condition implies in particular that for all projections to
affine Grassmannians, the corresponding affine Deligne-Lusztig variety is
non-empty, but is stronger than that. However, as the following example
shows, there are still more elements $x$ which give rise to an empty
affine Deligne-Lusztig variety.

\begin{example}
Let $G = SL_3$, $b = \epsilon^\lambda$ where $\lambda = (2,0,-2)$.  Let $x =
s_{01210120120} = \epsilon^{(3,1,-4)} s_{121}$ (we write $s_{12}$ for
$s_1s_2$ etc.). Then $x \ge b$ (a reduced expression for $b$ is
$s_{01210121}$), and $x{\bf a}$ is not a $P$-alcove for any proper parabolic
subgroup $P$. However, $X_x(b) = \emptyset$. (Cf.~Figure 3.24 in \cite{Reuman1} which
shows the situation for this $b$.)
\end{example}

\subsection{Non-emptiness of $X_x(b)$ for $b$ basic}

In this subsection, let $b$ be basic in $G(L)$.  In that case Lemma \ref{BMvsBG} and the injectivity of the left vertical arrow of (\ref{Newton_is_functorial_diagram}) imply the following:
if $[b] \cap M(L) \ne \emptyset$ for some semistandard Levi subgroup $M\subseteq G$,
then Lemma \ref{BMvsBG} shows that
$[b]\cap M(L)$ is a single $\sigma$-conjugacy class inside $M$ with the
same Newton vector as the Newton vector of $[b]$ with respect to $G$. (On
the other hand, the standard representative of $[b]$ with respect to $G$ is
not necessarily an element of $M$, and in particular is in general
different from the standard representative with respect to $M$.)

Applying Proposition \ref{nec_cond_prop} to the basic case, we get

\begin{cor}\label{ConsequBasic}
Let $[b]$ be basic.
Suppose $P=MN$ is a semistandard parabolic subgroup such that $x{\bf a}$ is a $P$-alcove.
Then $X_x(b)=\emptyset$,
unless $[b]$ meets $M(L)$ and $\eta_M(x) = \eta_M(\overline{\nu}_b)$.
\end{cor}

Let us emphasize that $\eta_M(\overline{\nu}_b)$ is really an abbreviation; here it stands for the value under $\eta_M$ for the unique $\sigma$-conjugacy class $[b'] \in B(M)$ which satisfies $\eta_G(b') = \eta_G(b)$ and $\overline{\nu}^M_{b'} = \overline{\nu}_b$.

\begin{conj} \label{conj2}
In the corollary, the opposite implication holds as well.
In other words, when $b$ is basic,  $X_x(b)$ is empty if and only if there
exists a semistandard
$P=MN$ such that
$x{\bf a}$ is a $P$-alcove, and $\eta_M(x) \ne
\eta_M(\overline{\nu}_b)$.
\end{conj}

This conjecture can be checked in the rank 2 cases ``by hand'', and in
higher rank cases, computer experiments provide further support for the
conjecture: it has been confirmed for the simply connected groups
(i.~e.~for $b=1$) of type $A_3$ and $x$ of length $\le 27$, of type $A_4$
and $x$ of length $\le 17$ and of type $C_3$ and $x$ of length $\le 23$,
and in several cases with $b$ basic, but different from $1$.

\medskip

In the remainder of this subsection we discuss some sufficient conditions
for the non-emptiness of
$X_x(b)$, when $b$ is basic.

\begin{lemma}
Let $x = \epsilon^\lambda w \in \widetilde{W}$ be an element which is not
contained in any Levi subgroup. Then
\[
X_x(b) \ne \emptyset \Longleftrightarrow \eta_G(x) = \eta_G(b).
\]
\end{lemma}

Here by \emph{not contained in any Levi subgroup}, we mean that no
representative of $x$ in $N_G(A)(L)$ is contained in a Levi subgroup of $G$
associated with a proper semistandard parabolic subgroup of $G$.
Since we consider only Levi subgroups containing the fixed maximal torus
$A$, their (extended affine) Weyl groups are subgroups of the (extended
affine) Weyl group of $G$. In terms of Weyl groups we can state the
condition as: the finite part $w$ of $x$ is not contained in any conjugate
of a proper parabolic subgroup of $W$.

If $w$ belongs to the Coxeter conjugacy class
 of $W$, then the condition is satisfied. For
the symmetric groups, i.~e.~if $G$ is of type $A_n$, the converse is also true,
as one sees using disjoint cycle decompositions.
 For all other types, however, there exist
other conjugacy classes which do not meet any (standard) parabolic subgroup
of $W$ (see for instance \cite{Geck-Pfeiffer}, where these conjugacy
classes are called cuspidal; some authors call them elliptic).

Before beginning the proof we note that similar considerations can be found in
\cite[Proposition 4.1]{kr03} and \cite[\S3.3.4]{Reuman1}.

\begin{proof}
As before, it is clear that $X_x(b)\ne\emptyset$ implies $\eta_G(x) = \eta_G(b)$.
On the other hand, given the latter condition, we will show that $x$ is
itself $\sigma$-conjugate to $b$, in other words that the Newton vector of
$x$ is $\overline{\nu}_b$. Our assumption ensures that $x$ is in the right connected
component of $G(L)$, so that we only need to prove that $x$ is basic.

In order to show that $x$ is basic, we prove that the Newton vector of $x$,
$\overline{\nu}_x = \frac{1}{N} \sum_{i=0}^{N-1} w^i
\lambda \in X_*(A)_{\mathbb Q}$ is $W$-invariant. (Here $N$ denotes the
order of $w$ in $W$.)
The point $\overline{\nu}_x$ lies in (the closure of) some Weyl chamber, and
hence its stabilizer is generated by a subset of the set of simple
reflections for this chamber, and hence is the Weyl group of some Levi
subgroup (or of all of $G$). On the other hand, $w$ is contained in this
stabilizer, and so our assumption gives us that the stabilizer of
$\overline{\nu}_x$ is in fact $W$.
\end{proof}

As the proof shows, if $G$ is semi-simple the elements $x\in\widetilde{W}$ which are not
contained in any Levi have finite order in $\widetilde{W}$.
Cf.~\cite{GHKR} Prop.~7.3.1.

Now let $x\in \widetilde{W}$. If $x$ is not contained in any Levi, then we
understand whether $X_x(b)=\emptyset$ by the lemma.
In general, there is a smallest semistandard Levi subgroup $M_-$ containing $x$,
and a smallest semistandard Levi subgroup $M_+ \supseteq M_-$ such that $x{\bf a}$ is a
$P_+$-alcove for some semistandard parabolic subgroup $P_+$ with Levi part $M_+$. Both of these
statements follow from \cite{Borel}, Prop.~14.22, which says that for (semistandard)
parabolic subgroups $P_1$, $P_2$, the subgroup $(P_1\cap P_2)R_uP_1$ is
again a (semistandard) parabolic subgroup; it has Levi part $M_1 \cap M_2$. There may be
more than one parabolic $P_+$ with Levi part $M_+$ for which $x{\bf a}$ is a $P_+$-alcove,
and of course, we may have $M_+ = P_+ = G$.

We then have, by Theorem \ref{HN}, (and assuming that $[b]$ meets $M_+$,
because otherwise $X^G_x(b) = \emptyset$, again by Theorem \ref{HN}),
\[
X^G_x(b) \ne \emptyset \Longleftrightarrow X^{M_+}_x(b) \ne \emptyset
\Longrightarrow \eta_{M_+}(x) = \eta_{M_+}(\overline{\nu}_b).
\]
Further, the lemma gives us (assuming that $[b]$ meets $M_-$)
\[
X^{M_-}_x(b) \ne\emptyset
\Longleftrightarrow \eta_{M_-}(x) = \eta_{M_-}(\overline{\nu}_b).
\]
The condition $\eta_{M_-}(x) =
\eta_{M_-}(\overline{\nu}_b)$ is quite restrictive; and
it becomes more restrictive the smaller $M_-$ is.

So, in terms of proving Conjecture \ref{conj2}, the case which remains to
consider is the case of $x$ which satisfy the following two conditions: (i) either $[b]$
does not meet $M_-$ or it does and $X^{M_-}_x(b) = \emptyset$, and (ii) $[b]$ meets
$M_+$ and $\eta_{M_+}(x) = \eta_{M_+}(\overline{\nu}_b)$.
The conjecture predicts that in this case
$X^{M_+}_x(b) \ne \emptyset$.

\subsection{Relation with Reuman's conjecture}\label{relation_with_Reumans_conj}

In this section, we will formulate a generalization of Reuman's conjecture,
and prove part of it, as a consequence of the results obtained above.
To formulate the conjecture, we consider the following maps from
$\widetilde{W}$ to $W$. The map $\eta_1$ is just the projection from
$\widetilde{W} = W \ltimes X_*(A)$ to $W$. It is a group homomorphism. To
describe the second map, we identify $W$ with the set of Weyl chambers. The
map $\eta_2 \colon \widetilde{W} \rightarrow W$ keeps track of the finite
Weyl chamber whose closure contains the alcove $x\mathbf a$.
We define $\eta_2(x) =
w$, where $w$ is the unique element in $W$ such that $w^{-1}x\mathbf a$ is
contained in the dominant chamber (so that the identity element of
$\widetilde{W}$ maps to the identity element of $W$).

We say that $x\in \widetilde{W}$ lies in the shrunken Weyl chambers, if
$k(\alpha, x\mathbf a) \ne k(\alpha, \mathbf a)$ for all roots $\alpha$, or
equivalently, if $U_\alpha \cap {^x}I \ne U_\alpha \cap I$ for all
$\alpha$.  For $T$ a subset of the set $S$ of simple reflections in $W$,
let $W_T \subset W$ denote the subgroup generated by $T$.  Let $\ell(w)$ denote the
length of an element $w \in \widetilde{W}$.  Finally, recall that we define the {\em
defect} ${\rm def}_G(b)$ of an element $b \in G(L)$ to be the $F$-rank of $G$ minus
the $F$-rank of $J_b$ (cf. \cite{GHKR}).

\begin{conj}\label{conj3}
a)
Let $[b]$ be a basic $\sigma$-conjugacy class.
Suppose $x\in \widetilde{W}$ lies in the shrunken Weyl chambers.
Then $X_x(b)\ne \emptyset$ if and only if
\[
\eta_G(x)=\eta_G(b), \text{ and }
\eta_2(x)^{-1}\eta_1(x)\eta_2(x) \in W\setminus \bigcup_{T\subsetneq S} W_T,
\]
and in this case
\[
\dim X_x(b) = \frac{1}{2}\left( \ell(x)
+\ell(\eta_2(x)^{-1}\eta_1(x)\eta_2(x)) - {\rm def}_G(b) \right).
\]
b)
Let $[b]$ be an arbitrary $\sigma$-conjugacy class, and let $[b_{\rm b}]$ be the
unique basic
$\sigma$-conjugacy class with $\eta_G(b) = \eta_G(b_{\rm b})$. Then there exists $N_b
\in \mathbb Z_{\ge 0}$, such that for all $x\in \widetilde{W}$ of length
$\ell(x) \ge N_b$, we have
\[
X_x(b) \ne \emptyset \Longleftrightarrow X_x(b_{\rm b}) \ne \emptyset,
\]
and in this case
\begin{equation*}
\dim X_x(b)=\dim X_x(b_{\rm b})-\frac{1}{2}\bigl(\langle 2\rho,\nu
\rangle+\defect_G(b)-\defect_G(b_{\rm b})\bigr),
\end{equation*}
where $\nu$ denotes the Newton point of $b$.
\end{conj}

Part (b) of this conjecture generalizes Conjecture 7.5.1 of \cite{GHKR}. It fits well
 with Beazley's Conjecture 1.0.1 and the
qualitative picture of $B(G)_x$ that is suggested by her results on $SL(3)$ (see
\cite{Beazley}).
The term
$\langle 2\rho,\nu
\rangle$ appearing here can also be interpreted (see section
\ref{sec.supset}) as the length of a suitable semistandard representative of $[b]$ in
$\widetilde W$.

Using the algorithms discussed in \cite{GHKR} and in this article, we
obtained ample numerical evidence for this conjecture. We made
computations  for root systems of type $A_2$, $A_3$, $A_4$, $C_2$, $C_3$,
$G_2$, and for a number of choices of $b$, including cases where $b$ is
split, basic, or neither of the two, and both cases where $\eta_G(b)=0$
and $\ne 0$.

The following remark shows that this conjecture is compatible
with what we already know about affine Deligne-Lusztig varieties in the affine
Grassmannian (cf. \cite{GHKR},\cite{V2}).

\begin{rem}
Conjecture \ref{conj3} implies Rapoport's dimension formula for affine
Deligne-Lusztig varieties $X_\mu(b)$ in the affine Grassmannian for $b$
basic (and $\mu \in X_*(A)$ dominant). Indeed, if $w_0 \in W$ is the longest
element, then we have
\[
\dim X_\mu(b) + \ell(w_0) = \sup \{\dim X_x(b);\ x\in W\epsilon^\mu W \}.
\]
Now for the longest element $x \in W\epsilon^\mu W$, we have $\eta_1(x) =
\eta_2(x) = w_0$, so $$\eta_2(x)^{-1}\eta_1(x)\eta_2(x) = w_0\in W\setminus
\bigcup_{T\subsetneq S} W_T,$$ and by the dimension formula given in the
conjecture, the supremum above is equal to
\[
\frac{1}{2}\left( \sup\{ \ell(x);\ x\in W\epsilon^\mu W \} + \ell(w_0)
-{\rm def}_G(b) \right).
\]
Let $X^\mu$ denote the $G(\mathfrak o)$-orbit of $\epsilon^\mu G(\mathfrak o)$
in the affine Grassmannian.
Since $$\sup\{ \ell(x);\ x\in W\epsilon^\mu W \} = \dim X^\mu + \ell(w_0) =
\langle 2\rho, \mu \rangle + \ell(w_0),$$ altogether we obtain
\[
\dim X_\mu(b) = \langle \rho, \mu \rangle - \frac{1}{2}{\rm def}_G(b),
\]
which is the desired result.
\end{rem}

Let us relate this conjecture to the results of the previous subsection. The
relation relies on the following lemma (which also follows easily from
Proposition \ref{acute_cone_prop}).

\begin{lemma}
Let $x\in\widetilde{W}$, and write $w = \eta_2(x)\in W$.

a) If $P=MN \supset {^w}B$ is a parabolic subgroup with
$x\in \widetilde{W}_M$, then $x{\bf a}$ is a $P$-alcove.

b)
If $x$ is an element of the shrunken Weyl chambers which is a $P$-alcove for
a semistandard parabolic subgroup $P$, then $P \supset {^w}B$.
\end{lemma}

\begin{proof}
First note that by assumption $w^{-1}x\mathbf a$ lies in the
dominant chamber. This means precisely that ${^{w^{-1}x}}I \cap U
\subseteq I \cap U$ (where $U$ denotes the unipotent radical of our Borel
$B$), so we obtain
\[
{^x}I \cap N \subseteq {^x I} \cap {^w U} \subseteq {^w} (I \cap U)
\subseteq I. \]
This inclusion is what we needed to show for part a).

Now let us prove b). Assume $x{\bf a}$ is a $P$-alcove and write $P=MN$ for the Levi decomposition of $P$.
We need to
show that $N \subseteq {^w}U$. Let $\alpha \in R_N$. Then we have
\[
{^x} I \cap U_\alpha \subsetneq I \cap U_\alpha.
\]
(We get $\subsetneq$ rather than just $\subseteq$ because $x$ is in the
shrunken Weyl chambers.) This implies however that
\[
{^x} I \cap U_{-\alpha} \supsetneq I \cap U_{-\alpha}.
\]
On the other hand, by what we have seen above,
\[
{^x}I \cap {^w}U \subseteq {^w}I \cap {^w}U \subseteq {^w}U(\epsilon \mathfrak o).
\]
This shows that $U_{-\alpha} \not\subseteq {^w}U$, hence $U_\alpha
\subseteq {^w} U$, as we wanted to show.
\end{proof}

{}From this lemma, we obtain the following strengthening of the ``only if'' direction
of part a) of Conjecture \ref{conj3} above.

\begin{prop} \label{Dynkin_conn}
Assume that the Dynkin diagram of $G$ is connected. Let $b$ be basic.
Let $x \in \widetilde{W}$, and write $x = \epsilon^\lambda v$, $v\in W$.
Assume that $\lambda \ne \overline{\nu}_b$ and that
$\eta_2(x)^{-1}\eta_1(x)\eta_2(x) \in
\bigcup_{T \subsetneq S}W_T$. Then $X_x(b) = \emptyset$.
\end{prop}

\begin{proof}
Write $w := \eta_2(x) \in W$.
By the lemma and our hypothesis, $x{\bf a}$ is a $P$-alcove for a proper parabolic subgroup $P =MN
\supset {^w}B$ of $G$. The only thing we need to check in order to apply
Corollary \ref{ConsequBasic} is that $\eta_{M'}({^{w^{-1}}}x) \ne
\eta_{M'}(\overline{\nu}_b)$, where $M' = {^{w^{-1}}}M$. (Recall that the precise meaning of $\eta_{M'}(\overline{\nu}_b)$ is described after Cor. \ref{ConsequBasic}.)  But if we
had equality here, then $w^{-1}\lambda - \overline{\nu}_b$ would be
a linear combination of coroots of $M'$. On the other hand,
$w^{-1}\lambda$ is dominant, and since $M'$ is the Levi component
of a proper standard parabolic subgroup, we obtain $\lambda =
\overline{\nu}_b$, which is excluded by assumption.
\end{proof}

Why does this imply the ``only if'' direction of part a) of Conjecture \ref{conj3}?  We need to show that $X^G_x(b) = \emptyset$ if $x{\bf a}$ is shrunken and $\eta_2(x)^{-1} \eta_1(x) \eta_2(x)$ belongs to a proper parabolic subgroup of $W$.  Let $G_i$ denote a simple factor of $G_{\rm ad}$, and let $x_i$ resp.~$b_i$ denote the image of $x$ resp.~$b$ in $G_i$.  Choose $i$ such that $\eta_2(x_i)^{-1} \eta_1(x_i) \eta_2(x_i)$ belongs to a proper parabolic subgroup of the Weyl group of $G_i$.  It is enough to prove that $X^{G_i}_{x_i}(b_i) = \emptyset$, since this obviously implies $X^G_x(b) = \emptyset$.  Therefore we can and shall assume that $G = G_i$, so that the Dynkin diagram of $G$ is connected, from now on.  Now write $x = \epsilon^\lambda v$.  We claim that if $x{\bf a}$ belongs to the shrunken
Weyl chambers and $\eta_2(x)^{-1}\eta_1(x)\eta_2(x)$ belongs to a proper parabolic
subgroup of $W$, then $\lambda \neq \overline{\nu}_b$.   Suppose instead that $\lambda
= \overline{\nu}_b$.  Then $\epsilon^\lambda$ belongs to the center of $G$ and $x{\bf
a} = v{\bf a}$.  This alcove belongs to the shrunken Weyl chambers only if $\eta_1(x)
= v = w_0$.  But in that case $\eta_2(x)^{-1} \eta_1(x) \eta_2(x)$ cannot belong to a
proper parabolic subgroup of $W$.  This proves our claim, and then we may apply Proposition \ref{Dynkin_conn} to conclude that $X^G_x(b) = \emptyset$.

\smallskip

We conclude this subsection by showing that our Conjecture \ref{conj2}
implies the validity of the ``if''  direction of part a) of Conjecture \ref{conj3}.

\begin{prop}\label{prop.crcon}
Assume that Conjecture \ref{conj2} holds. Let $x\in \widetilde{W}$ be an element
of the shrunken Weyl chambers with $\eta_G(x) = \eta_G(b)$ and
$$\eta_2(x)^{-1}\eta_1(x)\eta_2(x) \in
W\setminus \bigcup_{T \subsetneq S}W_T.$$ Then $X_x(b) \ne \emptyset$.
\end{prop}

\begin{proof}
It is enough to show that $x{\bf a}$ is not a $P$-alcove for any {\em proper} parabolic
subgroup $P = MN \subset G$. By the lemma above, if it were we would have
$P \supset {^{\eta_2(x)}}B$.  But the assumption says precisely that $x$
does not lie in $\widetilde{W}_M$ for such $P$.
\end{proof}

\section{Dimension theory for the groups $I_MN$} \label{dim_theory_section}
\subsection{}
In this section we lay some conceptual foundations for studying the dimensions of affine Deligne-Lusztig varieties $X_x(b)$, where $[b] \in B(G)$ is an arbitrary $\sigma$-conjugacy class.  These foundations play a key role in the sections that follow.  

We insert a remark about the notion of dimension: Using the usual definition of (Krull) dimension as the supremum of the lengths of chains of irreducible closed subsets, we can speak about the dimension of $X_x(b)$ without knowing anything about these subsets. Note though that we do know that they are schemes, locally of finite type, over $\overline{k}$ (see~\cite{Hartl-Viehmann}, Cor.~5.5), and that they are finite-dimensional (as follows from the corresponding result for affine Deligne-Lusztig varieties in the affine Grassmannian). In the proof below, it is however of crucial importance to work with the inverse image of $X_x(b)$ in $G(L)$, and to assign a ``dimension'' to this inverse image, and to more general (``ind-admissible'') subsets of $G(L)$.

In the case where $b = \epsilon^\nu$ for some $\nu \in X_*(A)$, a similar study was carried out in \cite{GHKR}, section 6.  The result was a finite algorithm to compute dimensions (a special case of our Theorem \ref{dim_alg} below).  In this paragraph, we introduce a suitable framework of ind-admissible sets and their dimension that works for general elements $b$.

Let $J$ be an Iwahori subgroup which is the fixer of an alcove in the standard
apartment, and let $P = MN \supset A$ be any parabolic
subgroup of $G$. Let $J_P = J_M N$ (where $J_M := J\cap M$). We will define the ind-admissible 
subsets of $J_P$ and then establish a ``dimension theory'' for them,
similar to the theory in \cite{GHKR}.  The groups $J_P$ ``interpolate'' between the extreme cases $I$ and $A(\mathfrak o)U(L)$, and as we will see they are precisely adapted to the study of affine Deligne-Lusztig varieties for elements $b$ more general than the extreme cases $b=1$ and $b$ a translation element.

Fix any semistandard Borel subgroup contained in $P$ and use it to define the sets of simple roots $\Delta_M$ and $\Delta_N$.
We fix a coweight $\lambda_0$ with $\langle \alpha,
\lambda_0\rangle = 0$ for $\alpha\in \Delta_M$, and $\langle \alpha,
\lambda_0\rangle > 0$ for $\alpha\in \Delta_N$, and consider the subgroups
\[
N(m) := \epsilon^{m\lambda_0} (N\cap J) \epsilon^{-m\lambda_0}, \quad
m\in\mathbb Z,
\]
cf.~loc.~cit.~5.2; our choice of $\lambda_0$ is a little different, but
this clearly does not affect the validity of the dimension theory for $N$
as in loc.~cit. Furthermore, we choose a separated descending filtration
$(J_{M}(m))_{m\in\mathbb Z}$ of $J_M$ by normal subgroups, such that
$J_{M}(m) = J_M$ for $m\le 0$, and such that all the quotients
$J_{M}(m)/J_{M}(m')$ are finite-dimensional over $\overline{k}$.
(For example, we could use a Moy-Prasad filtration.)  Finally, we set $J_P(m) := J_{M}(m) N(m)$, and we obtain a
separated and exhaustive filtration
\[
J_P \supset \cdots J_P(-1) \supset J_P(0) \supset J_P(1) \supset J_P(2) \supset \cdots.
\]
The quotients $J_P(m) / J_P(m')$, $m\le m'$ are finite-dimensional varieties
over $k$ in a natural way (more precisely, they coincide, in a natural way,
with the set of $\overline{k}$-valued points of a $k$-variety). Since $J_M$
normalizes each $N(m)$, $J_P(m) / J_P(m')$ is a fiber bundle over
$J_{M}(m)/J_{M}(m')$ with fibers $N(m)/N(m')$.
We say that a subset $Y \subseteq J_P$ is {\em admissible}, if there are $m \le
m'$ such that it is contained in $J_P(m)$ and is the full inverse image
under the projection $J_P(m) \rightarrow J_P(m) / J_P(m')$ of a locally
closed subset of $J_P(m) / J_P(m')$. We say that $Y \subseteq J_P$ is
{\em ind-admissible}, if for all $m$, $Y\cap J_P(m)$ is an admissible subset of
$J_P$. Obviously, admissible subsets are in particular ind-admissible.

As in \cite{GHKR}, for an admissible subset $Y \subset J_P(m)$, we can define a notion of dimension
$$
{\rm dim}\,Y := {\rm dim}(Y/J_P(m')) - {\rm dim}(J_P(0)/J_P(m'))
$$
for suitable $m' \geq 0$; note this is always an element of $\mathbb Z$, unless $Y$ is empty.  For an ind-admissible subset $Y \subset J_P$, we define
$$
{\rm dim} \, Y : = {\rm sup}\{{\rm dim}(Y \cap J_P(-m)) ~ : ~ m \geq 0 \}.
$$
We may sometimes have ${\rm dim} \, Y = +\infty$ (for example for $Y = J_P$).  Of course in making these definitions we made a choice, namely we normalized things so that ${\rm dim}(J_P(0)) = 0$.  But as before differences
$$
{\rm dim} \, Y_1 - {\rm dim} \, Y_2
$$
for admissible subsets $Y_1,Y_2$ are independent of any such choice.

\section{The generalized superset method} \label{superset_section}
\subsection{}
Recall that in \cite{GHKR}, Theorem 6.3.1, the dimension of $X_x(\epsilon^\nu)$ is expressed in terms of the dimensions of intersections of $^wU(L)$- and $I$-orbits in $G(L)/I$ (for $w \in W$).  Such intersections can be understood in terms of foldings in the Bruhat-Tits building of $G(L)$ (see loc.~cit.~6.1), and in this way we got an algorithm to compute ${\rm dim}~X_x(\epsilon^\nu)$. This algorithm led to and supported our conjectures in \cite{GHKR}.

In this section we explain the {\em generalized superset method}, which extends the above from translation elements $b = \epsilon^\nu$ to general $b$.  Correspondingly, it provides the data for the dimensions in the general case,  and is of independent interest because it shows that the emptiness patterns coincide in the $p$-adic and function field 
cases (see Corollary~\ref{functionfield_vs_padic}).  The generalized superset method involves the intersections of $^wI_P$- and $I$-orbits (for $w \in W$).  Such intersections can also be interpreted combinatorially in terms of foldings in the building.  For this we need to consider a new notion of retraction that is adapted to $I_P$-orbits rather than $U(L)$-orbits.  We will start with a discussion of these new retractions. 

\subsection{The retractions $\rho_P$}

Fix a standard parabolic $P=MN$. Write $I_P = I_MN = (I\cap M(L))N(L)$.

\begin{lemma} \label{IPGI}
Let $w\in \widetilde{W}$, and $J_P = {^{w^{-1}}}I_P$.
The projection $N_GA(L) \rightarrow J_P \backslash G(L)/I$ induces a
bijection
\[
\widetilde{W} \cong J_P \backslash G(L) /I.
\]
\end{lemma}

\begin{proof}
Because we can conjugate the situation by $w^{-1}$, we may as well assume
that $w=1$.
Since the set $P\backslash G(L)/K$ has only one element, we can identify
the double quotient $P\backslash G(L)/I$ with $W_M\backslash W \cong
\widetilde{W}_M \backslash \widetilde{W}$.
We obtain a commutative diagram
\[
\xymatrix{
    \widetilde{W} \ar[r]\ar[d]^q & I_P \backslash G(L) / I \ar[d]^p \\
    \widetilde{W}_M\backslash \widetilde{W} \ar[r]^{\cong} & P\backslash
    G(L) / I.
}
\]
Now for $v \in \widetilde{W}$, we have
\[
q^{-1}(\widetilde{W}_Mv)
= \widetilde{W}_M v
\cong I_M\backslash M/({^v}I)_M
\cong I_P\backslash P/({^v}I\cap P)
\cong p^{-1}(P v I).
\]
This proves the lemma.
\end{proof}

Denote by ${^M} W$ the set of minimal length representatives in $W$ of the
cosets in $W_M\backslash W$.

\begin{lemma}
Let $\lambda \in X_*(A)$ be such that $\langle \alpha, \lambda \rangle = 0$
for all roots $\alpha$ in $M$, and let $v \in {^M} W$.
\begin{enumerate}
\item All elements of $I_M$ fix the alcove $\epsilon^\lambda v\mathbf a$.
\item If $n\in N$, and if $\lambda$ satisfies $\epsilon^{-\lambda} n \epsilon^\lambda
\in {^vI}\cap N$ (which is true whenever $\lambda$
is sufficiently antidominant with respect to the roots in $\mathop{\rm
Lie} N$), then $n$ fixes the alcove $\epsilon^\lambda v \mathbf a$.
\end{enumerate}
\end{lemma}

\begin{proof}
To prove (1), we first note that $({^v}I)_M = I_M$, because $v$ is the
minimal length representative in its $W_M$-coset. This shows that
\[
I_M = {^{\epsilon^\lambda v}}(I \cap {^{v^{-1}}}M)  \subseteq {^{\epsilon^\lambda v}}I.
\]
Similarly, under the assumption on $n$ made in (2), we obtain that $n \in
 {^{\epsilon^\lambda v}}I$.
\end{proof}

Denote by $\mathcal A$ the standard apartment of $G$ with respect to our
fixed torus $A$.  Let $\rho_P$ be the retraction from the Bruhat-Tits building of $G(L)$ to $\mathcal A$,
defined as follows. For each alcove $\mathbf b$ in the
building, all retractions of $\mathbf b$ with respect to an alcove
of the form $\epsilon^\lambda v \bf a$, $\lambda$, $v$ as in part (2) of the lemma,
have the same image, say $\mathbf c$.  Here we must stipulate that $\lambda$ is sufficiently anti-dominant (depending on ${\bf b}$)
with respect to the roots in $\mathop{\rm Lie} N$.  We set $$\rho_P(\mathbf b) = \mathbf c.$$
(In fact, we get the same retraction if we retract with respect to any
alcove which lies between the root hyperplanes $H_\alpha$ and
$H_{\alpha,1}$ for all roots $\alpha$ of $M$, and is sufficiently
antidominant for all roots of $G$ lying in $N$. Compare also Rousseau's notion
of \emph{cheminée}, \cite{Rousseau} \S 9.)

\begin{lemma}
For $g \in I_P$, $\rho_P |_{g\mathcal A} = g^{-1}$.
\end{lemma}

\begin{proof}
Clearly, $g^{-1}$ maps $g\mathcal A$ to $\mathcal A$, and $g^{-1}$ fixes
the alcoves $t_\lambda v \mathbf a$ for $\lambda$ sufficiently
anti-dominant. This implies the lemma.
\end{proof}

The group $G(L)$ acts transitively on the set of extended alcoves, and the stabilizer of the base alcove is the Iwahori $I$. Therefore we can identify the quotient $G(L)/I$ with the set of extended alcoves.

\begin{prop} \label{I_P-orbits}
Let $y\in \widetilde{W}$.
\begin{enumerate}
\item
We have
\[
I_PyI/I = \rho_P^{-1}(y\mathbf a).
\]
In other words: we can identify $\rho_P$ (as a map from the set of alcoves
in the building to the set of alcoves in the standard apartment) with the
map $G(L)/I \rightarrow I_P\backslash G(L)/I \cong \widetilde{W}$ obtained
from Lemma \ref{IPGI}.
\item
More generally, let $w \in \widetilde{W}$, and let $J_P = {^{w^{-1}}}I_P$.
Consider the map
\[
\rho_{P,w} \colon G(L)/I \rightarrow \widetilde{W}, \qquad g
\mapsto w^{-1} \rho_P (wg).
\]
Then
\[
J_P y I = \rho_{P,w}^{-1}(y{\bf a}).
\]
\end{enumerate}
\end{prop}

\begin{proof}
Part (1) follows from the previous lemma, cf.~\cite{bruhat-tits72}, Remarque
7.4.22 which deals with the case $P=G$. To prove part (2), combine part (1)
with the following commutative diagram:
\[
\xymatrix{
G(L)\ar[r]^(.4){{\rm proj}} \ar[d]_{w^{-1}\cdot-} \ar@/^.8cm/[rr]^{\rho_P} & I_P\backslash
G(L)/I \ar[r]^(.6){\cong} \ar[d]_{w^{-1}\cdot -} & \widetilde{W}
\ar[d]_{w^{-1}\cdot-} \\
G(L)\ar[r]^(.4){{\rm proj}} & J_P\backslash G(L)/I \ar[r]^(.6){\cong} & \widetilde{W} \\
}
\]
\end{proof}

In the extreme cases, we get the following: If $P=G$, then $\rho_G$ is just
the usual retraction $\rho_{\mathbf a}$ with respect to the base alcove. If
$P=B$, then we get as $\rho_B$ the retraction with respect to ``a point at
infinity in the $B$-antidominant chamber''. Note that the maps $\rho_{P,w}$
are retractions to the standard apartment just like the $\rho_P$, but for a
different choice of base alcove.

\subsection{An algorithm for computing $\dim X_x(b)$}

In this subsection, we give a formula for the dimensions $$\dim X_x(b)\cap
I_Pw\mathbf a,$$ for any $w\in \widetilde{W}$. The method should be seen as an interpolation of the cases where $b$ is a translation element and $b=1$, respectively.  See Example~\ref{example.superset}, where we discuss how these extreme cases fit into the framework used here.

Let $[b]\in B(G)_P$. From the dimensions $\dim X_x(b)\cap I_Pw\mathbf a,$ we get the dimension
of $X_x(b)$, because we have
\begin{equation}\label{dim_as_sup}
\dim X_x(b) = \sup_{w\in \widetilde{W}} \dim (X_x(b) \cap I_Pw\mathbf a).
\end{equation}
To show this, observe that
\[
\dim X_x(b) = \sup_{v\in \widetilde{W}} \dim (X_x(b) \cap \overline{Iv\mathbf a}),
\]
where $\overline{\cdot}$ indicates the closure. Now every
$\overline{Iv\mathbf a}$ is contained in a finite union of $I_P$-orbits, in fact
\[
\overline{Iv\mathbf a} \subseteq \bigcup_{w\in S_v} I_Pw\mathbf a
\]
where $S_v := \{ w \in \widetilde{W} \, : \, w \leq v \}$.
Thus
$$\dim (X_x(b) \cap \overline{Iv\mathbf a}) = \sup_{w\in S_v} \dim
(X_x(b) \cap \overline{Iv\mathbf a} \cap I_Pw\mathbf a) \le \sup_{w\in
\widetilde{W}} \dim (X_x(b) \cap I_Pw\mathbf a)$$
 which shows that in
(\ref{dim_as_sup}), $\le$ holds. Since the inequality $\ge$ is obviously
true, the desired equality follows. Also note that we know a priori that
$\dim X_x(b)$ is finite, for example by using the finite-dimensionality of
affine Deligne-Lusztig varieties in the affine Grassmannian,
established in \cite{GHKR} and \cite{V1}.

Our result in Theorem \ref{dim_alg} is not
a ``closed formula'', even for fixed $w$, because it involves the
dimensions of intersections of $I$- and ${^{w^{-1}}}I_P$-orbits.
However, these dimensions can be computed (at least by a computer) for
fixed $w$.  (Here we make use of the interpretation of $I_P$-orbits in terms of ``foldings'', see Proposition \ref{I_P-orbits}.)

Throughout this subsection, we fix a $\sigma$-conjugacy class, say $[b] \in
B(G)_P \subset B(G)$, letting $M$ denote the Levi component of a standard parabolic
$P=MN$.  Denote by $b \in \widetilde{W}_M$ the standard representative of $[b]$
(see Definition \ref{defn_std_repr}).  Write $I_P = I_M N$. We have $b I_P
b^{-1} = I_P$. Denote by $\nu \in X_*(A)_{\mathbb Q}$ the Newton vector for
$b$ (where $b$ is considered as an element of $M(L)$). Since $b$ is $M$-basic, $\nu$ is ``central in $M$''
(and in particular $M$-dominant).  Let $\nu_{\rm dom}$ denote the unique $G$-dominant
element in the $W$-orbit of $\nu$.

For any $y \in \widetilde{W}$, we write ${\bf a}_y := y{\bf a}$.  Let $\rho \in X^*(A)_\mathbb Q$ denote the half-sum of the positive roots of $A$ in $G$.

\begin{theorem}\label{dim_alg}  Let $w\in \widetilde{W}$. Then writing $\tilde{b} = w^{-1}bw$,
and denoting by $\nu$ the Newton vector of $b$, we have
\[
\dim (X_x(b) \cap I_P w \mathbf a) = \dim (I \mathbf a_x \cap {^{w^{-1}}}I_P \mathbf
a_{\tilde{b}}) - \langle \rho, \nu + \nu_{\rm dom} \rangle.
\]
\end{theorem}

\begin{proof}
Fix a representative of $w$ in $N_{G}A(L)$ fixed by $\sigma$, and again
denote it by $w$.
Then multiplication by $w^{-1}$ defines a bijection
\[
X_x(b) \cap I_P \mathbf a_w \cong X_x(w^{-1}bw)\cap {^{w^{-1}}}I_P \mathbf
a,
\]
which preserves the dimensions.  Note that ${^{w^{-1}}}I_P := {^{w^{-1}}}(I_P)$ here.

We write $\tilde{b} = w^{-1}bw$, and consider the map
\begin{eqnarray*}
f_{\tilde{b}} \colon {^{w^{-1}}}I_P & \longrightarrow & {^{w^{-1}}}I_P, \\
g & \mapsto & g^{-1}\tilde{b}\sigma(g)\tilde{b}^{-1}.
\end{eqnarray*}
Let
\[
\widetilde{X_x(\tilde{b})} = \{ g\in G(L);\ g^{-1}\tilde{b}\sigma{g}\in IxI
\}.
\]
Then $\widetilde{X_x(\tilde{b})} \cap {^{w^{-1}}}I_P =
f_{\tilde{b}}^{-1}(IxI\tilde{b}^{-1}\cap {^{w^{-1}}}I_P)$, so
\[
X_x(\tilde{b}) \cap {^{w^{-1}}}I_P \mathbf a =
f^{-1}_{\tilde{b}}(IxI\tilde{b}^{-1}\cap {^{w^{-1}}}I_P)/(I\cap
{^{w^{-1}}}I_P).
\]

\begin{lemma} \label{rel_dim_lemma}  We have the equality
\[
\dim f_{\tilde{b}}^{-1}(IxI\tilde{b}^{-1} \cap {^{w^{-1}}}I_P) -
\dim (IxI\tilde{b}^{-1} \cap {^{w^{-1}}}I_P) = \langle \rho, \nu - \nu_{\rm
dom} \rangle.
\]
\end{lemma}

\begin{proof}[Proof of Lemma]
To ease the notation, let us write $J_P:= {^{w^{-1}}}(I_P) = ({^{w^{-1}}I})_{{^{w^{-1}}P}}$, and
$J_M := ({^{w^{-1}}}I)_{{^w M}}$.  It is easy to see that $IxI\tilde{b}^{-1} \cap J_P$ is an admissible subset of $J_P$.  It will follow from our proof below that its preimage under $f_{\tilde b}$ is ind-admissible, so that we can define the dimensions of these subsets using the theory from section \ref{dim_theory_section}.  The left hand side of the equality is therefore well-defined.  We can even make a very convenient choice of filtration on $J_M$, one which is stable under ${\rm Ad}(\tilde b)$: take the Moy-Prasad filtration $J_{M}(\bullet)$ on $J_M$ associated to the barycenter of the alcove in the reduced building of $M(L)$ which corresponds to $J_M$.

A straightforward calculation shows that we can write the map $f_{\tilde{b}}$
as follows (here $i \in J_M$, $n \in {^{w^{-1}}}N$):
\[ g = in \mapsto g^{-1}\tilde{b}\sigma(g)\tilde{b}^{-1} =
i^{-1} \ {^{\tilde{b}}} \sigma(i) \cdot {^{\tilde{i}}}n^{-1}\
{^{\tilde{b}}}\sigma(n), \]
with $\tilde{i} := {^{\tilde{b}}}\sigma(i)^{-1}i$.

The projection $J_P \rightarrow J_M$ is an ``ind-admissible fiber bundle'', in a sense which the reader will have no trouble making precise (see section \ref{dim_theory_section}).  The above description of $f_{\tilde b}$ indicates how it behaves on the base and on the fibers.  Let us analyze the relative dimension of $f_{\tilde b}$ by studying the base and the fibers in turn.

First, we consider the base $J_M$.  Since $\tilde{b}$ normalizes $J_M$,
the map $J_M \rightarrow J_M$, $i \mapsto i^{-1}\
{^{\tilde{b}}} \sigma(i)$ is surjective, and has relative dimension zero.  The proof is an adaptation of the proof of Lang's theorem.  Indeed, $J_M$ has a filtration by normal subgroups (the $J_{M}(m)$ for $m \geq 0$ in the Moy-Prasad filtration described above) which are stabilized by ${\rm Ad}(\tilde b)$, such that on the finite-dimensional quotients our map $J_M \rightarrow J_M$ induces a Lang map, which is finite \'etale and surjective.

Second, we study the relative dimension of $f_{\tilde b}$ ``on the fibers'' of $J_P \rightarrow J_M$.  That is, we fix $\tilde i \in J_M$ as above, and study the fibers of the map ${^{w^{-1}}N(L)} \rightarrow {^{w^{-1}}N(L)}$ given by
$n \mapsto {^{\tilde i}n^{-1}} \,{^{\tilde b}}\sigma(n)$.
Fortunately, most of the necessary work was already done in \cite{GHKR}, Prop.~5.3.2. In fact, that
proposition implies that the fiber dimension is (using the notation of
loc.~cit.)
\[
d(\tilde{i}, \tilde{b}) := d(\mathfrak n(L), \mathop{\rm Ad}\nolimits_{\mathfrak
n}(\tilde{i})^{-1}\mathop{\rm Ad}\nolimits_{\mathfrak n}(\tilde{b})\sigma) + \mathop{\rm
val} \det \mathop{\rm Ad}\nolimits_{\mathfrak n}(\tilde{i}).
\]
Here $\mathfrak n$ denotes the Lie algebra of ${^{w^{-1}}}N$.
Since $\tilde{i} \in J_M$, the second summand vanishes.
Moreover, $\mathop{\rm Ad}\nolimits_{\mathfrak
n}(\tilde{i})^{-1}\mathop{\rm Ad}\nolimits_{\mathfrak n}(\tilde{b}) = \mathop{\rm
Ad}\nolimits_{\mathfrak n}(i^{-1}\tilde{b}\sigma(i))$. Since
$\sigma$-conjugation induces an isomorphism of $F$-spaces, we obtain
\[
d(\tilde{i}, \tilde{b}) = d(1, \tilde{b}) = \langle \rho, \nu
-\nu_{\rm dom} \rangle,
\]
cf.~loc.~cit.~Prop.~5.3.1.

It is clear that we should be able to put these two pieces of information together (and
obtain the stated result that the relative dimension of $f_{\tilde b} $ is $\langle \rho, \nu - \nu_{\rm dom}
\rangle$) by looking at the corresponding
finite-dimensional situation.  However, to make this vague idea convincing it seems easiest to follow the argument
of ~loc.~cit.~Prop.~5.3.1.  First, we correct for the inconvenient fact that $f_{\tilde b}$ need not preserve $J_P(0)$.  Let $P' := {^{w^{-1}}P}$,  $M' := {^{w^{-1}}M}$, $N' := {^{w^{-1}}N}$, and $I' := {^{w^{-1}}I}$.   For any $m_1,m_2 \in M'(L)$ which normalize $J_P = I'_{P'}$, define
\begin{eqnarray*}
f_{m_1, m_2} \colon J_P & \longrightarrow & J_P, \\
g & \mapsto & m_1g^{-1}m_1^{-1} \cdot m_2 \sigma(g) m_2^{-1}.
\end{eqnarray*}
Note that $f_{\tilde b} = f_{1, \tilde b}$.  Fix $\lambda_0 \in X_*(Z(M'))$ such that $\langle \alpha, \lambda_0 \rangle > 0$ for all $\alpha \in R_{N'}$.  Then we may replace $f_{\tilde b} = f_{1, \tilde b}$ with $f := f_{\epsilon^{t\lambda_0}, \epsilon^{t\lambda_0}\tilde b}$ for a suitably large integer $t$, chosen such that
$f$ preserves $J_P(0) = I'_{M'}\cdot N' \cap I'$.  Note that $f$ then automatically preserves $J_P(m)$ for each integer $m \geq 0$ (we shall not need this fact).   Denote by $f_0: J_P(0) \rightarrow J_P(0)$ the restriction of $f$ to $J_P(0)$.  As in ~loc.~cit., our goal is now to prove the following

 \smallskip

{\bf Claim:}  Let $m_1 = \epsilon^{t\lambda_0}$ and $m_2 = \epsilon^{t\lambda_0}\tilde b$ and set $f := f_{m_1,m_2}$.  If $Y \subset J_P$ is admissible, then $f^{-1}Y $ is ind-admissible and
$$
{\rm dim} \, f^{-1} Y - {\rm dim} \, Y = d(m_1,m_2).
$$

Continuing to follow the strategy of the proof of Prop.~5.3.2 of ~loc.~cit., we can use {\em the proof of} ~loc.~cit.~Claim 1 to find an $a := \epsilon^{t_1\lambda_0}$ for a large integer $t_1$ such that
$$
c_aJ_P(0) \subseteq fJ_P(0),
$$
where $c_a$ denotes the conjugation map $g \mapsto aga^{-1}$ for $g \in J_P$.  Fix this element $a$ once and for all.  Next we prove the following

\smallskip

{\bf Subclaim:}  Suppose that $Y$ is an admissible subset of $c_aJ_P(0)$.  Then $f_0^{-1}(Y)$ is admissible, and
$$
{\rm dim} \, f_0^{-1}Y - {\rm dim} \, Y = d(m_1,m_2).
$$

\noindent {\em Proof of Subclaim:}  At this point we have to replace the filtration $\{J_P(m)\}_{m \geq 0}$ of $J_P(0)$ with one which is better behaved with respect to the morphism $f_0$.  So, for $m \geq 0$ let $I'_m \subset I'$ denote the $m$-th principal congruence subgroup of the Iwahori subgroup $I'$; by convention $I'_0 = I'$.  Let $J_{M,m} := I'_{m} \cap M'$ and $N'_m := I'_m \cap N'$.  Let $J_{P,m} = J_{M,m} N'_m = I'_m \cap P'$.  It is clear that $J_M$ normalizes each $N'_m$, so that we have a fiber bundle for each $0 \leq m_1 \leq m_2$
$$
\pi: J_{P,m_1}/J_{P,m_2} \rightarrow J_{M,m_1}/J_{M,m_2}
$$
with fiber $N_{m_1}/N_{m_2}$.  Also, using our specific choices of $m_1,m_2$ above, it is clear that $f_0$ preserves $J_{P,m}$ and in fact $f_0$ induces a well-defined map on the quotients
$$
\overline{f}: J_{P,0}/J_{P,m} \rightarrow J_{P,0}/J_{P,m}
$$
for any $m \geq 0$.  Here, we used that $m_1$ and $m_2$ and $J_{P,0}$ each normalize $J_{P,m}$, for all $m \geq 0$. (See (\ref{condition(2'')}).)

Now choose a large positive integer $m$ such that $Y$ comes from a locally closed subset $\overline{Y}$ of $J_{P,0}/J_{P,m}$.  Consider the following commutative diagram
$$
\xymatrix{
J_{P,0} \ar[d]_p \ar[r]^{f_0} & J_{P,0} \ar[d]_p \\
J_{P,0}/J_{P,m} \ar[d]_\pi \ar[r]^{\overline{f}} & J_{P,0}/J_{P,m} \ar[d]_\pi \\
J_{M,0}/J_{M,m} \ar[r]^{\overline{f}_M} & J_{M,0}/J_{M,m},
}
$$
where $p$ is the canonical projection, $\pi$ is the fiber bundle described above, and $\overline{f}$ and $\overline{f}_M$ are the morphisms induced by $f_0$.  Note that
$f_0^{-1} Y = p^{-1}\overline{f}^{-1}\overline{Y}$, showing that $f_0^{-1} Y$ is admissible.  Note also that since $Y \subseteq c_aJ_P(0) \subseteq fJ_P(0)$, the subset $\overline{Y}$ is contained in the image of $\overline{f}$, and our dimension formula is a consequence of the identity
$$
{\rm dim} \, \overline{f}^{-1}\overline{Y} - {\rm dim} \, \overline{Y} = d(m_1, m_2).
$$
But the latter equality now follows easily from our earlier considerations of the base and fiber of the fiber bundle $\pi$: the map $\overline{f}_M$ is surjective of relative dimension zero, and the relative dimension of $\overline{f}$ on locally closed subsets of the fibers of $\pi$ over $\pi(\overline{Y})$ is given by $d(m_1,m_2)$; see the proof of  ~loc.~cit.~Claim 3.  This proves our subclaim.

As in loc.~cit., our claim follows from the subclaim.  Write $d(m_1,m_2) =: d$.  If $Y \subset J_P$ is any admissible subset, then we have proved that $f^{-1} Y \cap a_1{^{ -1}}J_P(0) a_1$ is admissible of dimension ${\rm dim}\, Y + d$ for any $a_1 \in Z(M')(F)$ such that $a_1Y a_1^{ -1} \subseteq aJ_P(0) a^{-1}$.  Let $t_0$ be sufficiently large so that $a_t := \epsilon^{t\lambda_0}$ satisfies $a_t Y a_t^{-1} \subseteq aJ_P(0)a^{-1}$ for all $t \geq t_0$.  For all such $t$ we have proved that $f^{-1} Y \cap a_t^{-1} J_P(0) a_t$ is admissible of dimension ${\rm dim} \, Y + d$.  This is enough to prove the claim, hence also the lemma.    \end{proof}

\begin{rem} \label{f_b_surjective}
The proof of Lemma \ref{rel_dim_lemma} shows that $f_{\tilde b}: J_P \rightarrow J_P$ is {\em surjective}.
\end{rem}

Now let
\[
d(x, \tilde{b}, {^{w^{-1}}} I_P ) := \dim (I \mathbf a_x \cap {^{w^{-1}}}
I_P \mathbf a_{\tilde{b}}).
\]
We have a dimension-preserving bijection $$I\mathbf a_x \cap
{^{w^{-1}}}I_P\mathbf a_{\tilde{b}} \cong (IxI\tilde{b}^{-1}\cap
{^{w^{-1}}}I_P)/({^{w^{-1}}}I_P\cap {^{\tilde{b}}}I)$$
given by right
multiplication by $\tilde{b}^{-1}$, so that
\[
d(x,\tilde{b}, {^{w^{-1}}}I_P) = \dim IxI\tilde{b}^{-1} \cap {^{w^{-1}}}I_P
- \dim {^{w^{-1}}}I_P\cap {^{\tilde{b}}}I.
\]

Let $\rho_N \in X^*(A)_\mathbb Q$ denote the half-sum of the roots in $R_N$.

\begin{lemma}
Consider $c_{\tilde{b}} \colon {^{w^{-1}}}I_P \rightarrow {^{w^{-1}}}I_P$,
$g\mapsto \tilde{b}g\tilde{b}^{-1}$. Then $${^{w^{-1}}}I_P\cap
{^{\tilde{b}}}I = c_{\tilde{b}}({^{w^{-1}}}I_P\cap I),$$
hence
\[
\dim ({^{w^{-1}}}I_P\cap I) - \dim ({^{w^{-1}}}I_P\cap {^{\tilde{b}}}I) =
\langle 2\rho_N, \nu \rangle.
\]
\end{lemma}

\begin{proof}
As the previous lemma, this can be proved by looking at the projection $J_P \rightarrow J_M$
and then separately computing the contribution from the base $J_M$ (which is $0$) and that from
the fibers (which is $\langle 2\rho_N, \nu \rangle$, see \cite{GHKR}).
\end{proof}

Altogether we have now
\begin{eqnarray*}
&& \dim X_x(b) \cap I_P \mathbf a_{w} \\
& = & \dim f_{\tilde b}^{-1} (IxI\tilde{b}^{-1} \cap {^{w^{-1}}}I_P) - \dim I \cap {^{w^{-1}}}I_P \\
& = & \dim IxI\tilde{b}^{-1} \cap {^{w^{-1}}}I_P - \dim I\cap
{^{w^{-1}}}I_P + \langle \rho, \nu - \nu_{\rm dom} \rangle \\
& = & d(x,\tilde{b}, {^{w^{-1}}}I_P) + \dim {^{w^{-1}}}I_P \cap
{^{\tilde{b}}}I - \dim I\cap {^{w^{-1}}}I_P + \langle\rho, \nu- \nu_{\rm
dom}\rangle \\
& = & d(x,\tilde{b}, {^{w^{-1}}}I_P) + \langle\rho, \nu- \nu_{\rm
dom}\rangle - \langle 2\rho_N, \nu \rangle\\
& = & d(x,\tilde{b}, {^{w^{-1}}}I_P) - \langle \rho, \nu + \nu_{\rm dom}
\rangle,
\end{eqnarray*}
where in the final step we have used the equality $\langle \rho, \nu \rangle = \langle \rho_N, \nu \rangle$.  This is what we wanted to show.
\end{proof}

Together with the description (Proposition \ref{I_P-orbits}) of ${^{w^{-1}}}I_P$-orbits in $G(L)/I$ as
fibers of a certain retraction of the building, Theorem \ref{dim_alg} gives us an
algorithm to compute whether for a given $w$ the intersection $X_x(b) \cap
I_P w\mathbf a$ is empty or non-empty; compare \cite{GHKR} 6.1. If this information were available
for all $w$, we could conclude whether $X_x(b)$ is non-empty (and compute
its dimension from the dimensions of all these intersections). As noted above, it is clear
that all affine
Deligne-Lusztig varieties are finite-dimensional, so that the supremum of
$\dim (X_x(b) \cap I_P w\mathbf a)$ is attained for some $w$. It does not
seem easy to give a bound for the length of $w$ depending on $x$ and $b$.

The theorem allows us to compare the function field case with the $p$-adic case. For $b\in \widetilde{W}$, similarly as the $X_x(b)$ defined above, we have an ``affine Deligne-Lusztig set'' $X_x(b)_{\mathbb Q_p}$ inside $G(\widehat{\mathbb  Q}_p^{\rm ur})/I$, where $I$ denotes the corresponding Iwahori.

\begin{cor}\label{functionfield_vs_padic}
Let $x\in\widetilde{W}$ and $b\in \widetilde{W}$. Then $X_x(b)\ne\emptyset$ if and only if $X_x(b)_{\mathbb Q_p}\ne\emptyset$.
\end{cor}

\begin{proof}
One checks that, as far as the non-emptiness is concerned, the proof of Theorem~\ref{dim_alg} works without any changes in the $p$-adic case. The combinatorial properties of the retractions which describe the intersections occurring there coincide in the function field case and the $p$-adic case.
\end{proof}

Even for the dimensions, it is plausible to expect that arguments as in the proof of Theorem~\ref{dim_alg} can be used in the $p$-adic case, \emph{once a viable notion of dimension has been defined}.

\medskip

\begin{example}\label{example.superset}
As examples, let us consider the extreme cases:
\begin{enumerate}
\item
$P=B$. Then $I_P = A(\mathfrak o)U$, and $b = \epsilon^\nu \in B(G)_B$ where $\nu \in X_*(A)$ is a regular
dominant
translation element. This case
was considered in \cite{GHKR}. 
The above formula is the same as in
loc.~cit., equations (6.3.3), (6.3.4).
\item
$P=G$. Then $I_P = I$, and $b \in \Omega_G$ is a basic
$\sigma$-conjugacy class. In this case, the dimension formula reads
\[
\dim X_x(b) \cap Iw\mathbf a = \dim I\mathbf a_x \cap {^{w^{-1}}}I \mathbf
a_{w^{-1}bw}
\]
(since $\nu$ is central in $G$). This
case is the case analyzed by Reuman in \cite{Reuman2} for the
case $b=1$, and low-rank groups. So let $b=1$ (the case of other basic
$b$'s is analogous). We have that
\begin{eqnarray*}
X_x(1) \ne \emptyset
& \Longleftrightarrow & \exists w\in \widetilde{W} : IxI \cap {^{w^{-1}}} I
 \ne \emptyset \\
 &\Longleftrightarrow & \exists w\in\widetilde{W}: \rho_{G}^{-1}(x) \cap \rho_{G, w}^{-1}(1) \ne \emptyset.
\end{eqnarray*}
There are two ways to reformulate this. The algorithmic description in the
spirit of the above amounts to
\[
X_x(1) \ne \emptyset \Longleftrightarrow \exists w\in \widetilde{W} : 1 \in
\rho_{G,w}(IxI)
\]
On the other hand, we also obtain
\[
X_x(1) \ne \emptyset \Longleftrightarrow \exists w \in \widetilde{W} : x
\in \rho_{G}(Iw^{-1}IwI).
\]
which leads to the ``folding method'' used by Reuman, since $Iw^{-1}IwI/I$,
as a set of alcoves in the building, is exactly the set of alcoves which
can be reached by a gallery of type $i_r,\dots, i_1, i_1, \dots, i_r$ (for
a fixed reduced expression $w = s_{i_1}\cdots s_{i_r}$). See also section
\ref{sec.supset}.
\end{enumerate}
\end{example}

\begin{rem} dimension formula in Example \ref{example.superset}~(2) can be interpreted in terms of structure constants for the affine Hecke algebra.
Let $H$ denote the affine Hecke algebra over $\mathbb Z[v,v^{-1}]$ corresponding to
the extended affine Weyl group $\widetilde{W}$ and let $T_x \in H$ denote the
standard basis element corresponding to $x \in \widetilde{W}$.  Define the parameter
$q := v^2$, and consider the structure constants $C(x,y,z) \in \mathbb Z[q]$ for
$x,y,z \in \widetilde{W}$ defined by the equality in $H$
$$
T_xT_y = \sum_{z} C(x,y,z) T_z.
$$
Then it is straightforward to check that
\[
\dim I\mathbf a_x \cap {^{w^{-1}}}I \mathbf
a_{w^{-1}bw} = {\rm deg}_q C(x,w^{-1}b^{-1}, w^{-1}).
\]
(By convention, we set ${\rm deg}_q 0 := -\infty = \dim \emptyset$.)  Determining the structure constants
$C(x,w^{-1}b^{-1},w^{-1})$ is also a ``folding algorithm'', so this does not give an essentially different way to compute dimensions of affine Deligne-Lusztig varieties.  But it does give some insight on the inherent complexity of the algorithm.
\end{rem}

\section{On reduction to the basic case and a finite algorithm} \label{reduction_to_basic_section}
\subsection{}
One drawback of Theorem \ref{dim_alg} is that it does not produce a {\em finite} algorithm to compute the non-emptiness or dimension of $X^G_x(b)$.  In this section, we explain how we can at least find a finite algorithm which reduces the non-emptiness and dimension of $X^G_x(b)$ to that of a finite number of related varieties $X^{M'}_y(\tilde b)$, where for all the latter $\tilde b$ is basic in $M'$.

Using Theorem \ref{dim_alg}, we will usually have to check an infinite number
of orbit intersections to determine whether a given $X_x(b)$ is empty or not.  However, for $b$ basic, we have proved the emptiness predicted by Conjecture \ref{conj2} in Corollary \ref{ConsequBasic}.  Why are we confident that Conjecture \ref{conj2} also correctly predicts non-emptiness?  In order to confirm the non-emptiness of $X_x(b)$ in a case it is expected, it is sufficient for the computer to detect a single non-empty intersection $I\mathbf a_x \cap {^{w^{-1}}}I \mathbf
a_{w^{-1}bw}$ for some $w$, and in practice the computer does detect one (as far as we have checked).  In other words, concerning the non-emptiness question for $b$ basic, in practice the algorithm always terminates in finitely many steps, and in this way we are able to generate a complete emptiness/non-emptiness picture, at least when $\ell(x)$ is small enough for the computer to handle.

Let $P = MN$ denote a standard parabolic subgroup.  Suppose $b \in \Omega_M \subset M(L)$ is the standard representative of a basic $\sigma$-conjugacy class in $M(L)$, and let $\nu = \overline{\nu}^M_b$ denote its Newton vector.

Recall that $^MW$ denotes the set of minimal length representatives of the cosets in $W_M\backslash W$.  Note that $P \backslash G(L) /I \cong \, ^MW$.

  {}From now on, we fix an element $w \in \, ^MW$.  Write $M' = \,
^{w^{-1}}M$,
$N'= \,^{w^{-1}}N$, and $P' = \, ^{w^{-1}}P$.  Let us denote $\tilde b := \,
^{w^{-1}}b \in \Omega_{M'}$.  Note that $I_{M'} = \, ^{w^{-1}}(M \cap \,
^wI)  = \, ^{w^{-1}}(M \cap I)$ is an Iwahori subgroup of $M'$.  Let $e_0$
denote the base point of the affine flag variety $G(L)/I$ and let $e'_0$
denote the base point in $M'(L)/I_{M'}$.

We consider the map
\begin{align*}
\alpha_w: Pwe_0 &\rightarrow M'(L)/I_{M'}\\
mnwe_0 &\mapsto \, ^{w^{-1}}m e'_0,
\end{align*}
which is easily seen to be well-defined and surjective.  Fix $m \in M(L)$ and write $m' := \, ^{w^{-1}}m \in M'(L)$.  The map $mnwe_0 \mapsto \, ^{w^{-1}}n $ determines a bijection
\begin{equation} \label{alpha_w_fiber}
\alpha_w^{-1}(m'e_0') = \, N'/N' \cap I.
\end{equation}
We warn the reader that $\alpha_w$ is not a morphism of ind-schemes; however its restriction to the inverse image of any connected component of $M'(L)/I_{M'}$ is a morphism of ind-schemes.

Now for $x \in \widetilde{W}$, and $w,b$ as above, define the finite set
$$
S_P(x,w) := \{ y \in \widetilde{W}_{M'} ~ : ~ N' {\bf a}_y \cap I {\bf a}_x \neq \emptyset \}.
$$
Note that $N' {\bf a}_y \cap I {\bf a}_x \neq \emptyset \Leftrightarrow I_{P'}{\bf a}_y \cap I {\bf a}_x \neq \emptyset$.  For a given $x$, there are only finitely many $y$ such that the latter holds; see Proposition \ref{I_P-orbits}.

The following proposition is an analogue of part of \cite{GHKR}, Prop.~5.6.1.

\begin{prop}\label{prop.&&&}
\begin{enumerate}
\item[(1)]  The map $\alpha_w$ restricts to give a surjective map
\begin{equation} \label{beta_w}
\beta_w: X^G_x(b) \cap Pwe_0 \longrightarrow \bigcup_{y \in S_P(x,w)} X^{M'}_{y}({\tilde b}).
\end{equation}
\item[(2)]  Assume $X^G_x(b) \cap Pwe_0 \neq \emptyset$.
For a fixed $m' \in M'(L)$ such that $m'e_0' \in X^{M'}_y(\tilde b)$, set $b' :=
m'^{-1} \tilde b \sigma (m') \in I_{M'} y I_{M'}$.  Then the fiber
$\beta_w^{-1}(m'e_0')$ is a locally finite-type algebraic variety having dimension
$$
\dim \beta_w^{-1}(m'e_0') = \dim(I{\bf a}_x \cap \, N'{\bf a}_{y}) - \langle \rho, \nu + \nu_{\rm dom} \rangle,
$$
a number which depends on $y$ but not on $m'e'_0$.
\item[(3)]  We have
$$
\dim \, X^G_x(b) = \underset{w,y\, :\, y \in S_P(x,w)}{\rm sup}\{ \dim(I{\bf a}_x \cap \, ^{w^{-1}}N{\bf a}_{y}) + \dim (X^{\,^{w^{-1}}M}_{y}(\, ^{w^{-1}}b)) \}  - \langle \rho, \nu + \nu_{\rm dom} \rangle.
$$
\end{enumerate}
\end{prop}

The proposition implies that, modulo knowledge of
certain basic cases (i.e., the $X^{M'}_{y}(\tilde b)$), there is a finite algorithm
to determine the non-emptiness and dimension of $X^G_x(b)$.  Conjecture \ref{conj2}
predicts a finite algorithm to determine the non-emptiness of each $X^{M'}_y(\tilde
b)$.  Thus, in effect it predicts a finite algorithm for the non-emptiness of
$X^G_x(b)$ itself.

\begin{cor}
We have $X^G_x(b) \neq \emptyset$ if and only if there exist $w \in \, ^MW$ and $y \in S_P(x,w)$ with $X^{M'}_y(\tilde b) \neq \emptyset$.
\end{cor}

\noindent {\em Proof of  Proposition:}  It is clear that $\alpha_w$ sends the left hand side of (\ref{beta_w}) into the right hand side.  If $m'e'_0 \in X^{M'}_y(\tilde b)$, then the isomorphism (\ref{alpha_w_fiber}) restricts to give an isomorphism
\begin{equation} \label{beta_w_fiber}
\beta_w^{-1}(m'e'_0) = f_{b'}^{-1}(IxIb'^{-1} \cap N') / N' \cap I,
\end{equation}
where $b' := m'^{-1} {\tilde b}\sigma(m')$ and where we define
\begin{align*}
f_{b'} \colon & N' \longrightarrow  N'\\
& n' \longrightarrow  n'^{-1} b' \sigma(n') b'^{-1}.
\end{align*}
Since $f_{b'}$ is surjective (see Remark \ref{f_b_surjective}) and $IxI \cap N'b' \neq \emptyset$, we see that $\beta_w$ is surjective, proving (1).   Also, the fibers of $\beta_w$ are algebraic varieties locally of finite type, and their dimension can be computed from (\ref{beta_w_fiber}) using the method of the proof of Theorem \ref{dim_alg}.  This proves (2).  Finally, (3) follows from (1) and (2). \qed

\begin{rem} For affine Deligne-Lusztig varieties in the affine Grassmannian,
it is known that $X^G_\mu(b) \neq \emptyset$ if and only if $[b] \in B(G,\mu)$
\textup(cf.
\cite{kr03},\cite{Kottwitz-HN},\cite{Lucarelli},\cite{Gashi},\cite{W}\textup).  The
condition
$[b] \in B(G,\mu)$ means that $\eta_G(b) = \mu $ in $\Lambda_G$ and $\overline{\nu}_b
\leq \mu$ (``Mazur's inequality'').  For $X^G_x(b)$, where as before we take $b \in
\Omega_M$, one might ask for the analogues of ``Mazur's inequalities,'' where by this
we mean a family of congruence conditions and inequalities imposed on $x$,$b$ and
$\overline{\nu}_b$ which hold if and only if $X^G_x(b)$ is non-empty.   In light of
the above proposition, we see that, whatever Mazur's inequalities end up being, they
should hold if and only if there exists
$w \in \,^MW$ such that for some $y \in \widetilde{W}_{\,^{w^{-1}}M}$, we have
\begin{itemize}
\item $^{w^{-1}}Ny \cap IxI \neq \emptyset$ and
\item $X^{\, ^{w^{-1}}M}_y(\, ^{w^{-1}}b) \neq \emptyset$.
\end{itemize}
In view of Conjecture \ref{conj2}, the second item
should be understood as a family of congruence conditions.  The first item should
correspond to a family of inequalities and congruence conditions between $x,y \in \widetilde{W}$.  Taken
together the inequalities will be somewhat stronger than the condition $y \leq x$ in the Bruhat
order on $\widetilde{W}$.
\end{rem}

\hyphenation{semistandard}

\section{Fundamental alcoves and the superset method}\label{sec.supset}
\subsection{Fundamental alcoves}
We now single out some
alcoves that will be used to generalize Reuman's superset method \cite{Reuman2} to all
$\sigma$-conjugacy classes in $G(L)$.

\begin{defn}\label{def.fa}
For $x \in \widetilde W$ we say that $x\mathbf a$ is a \emph{fundamental
alcove} if every element of $IxI$ is $\sigma$-conjugate under $I$ to $x$.
\end{defn}
Equivalently, the alcove $x\mathbf a$ is fundamental if every element of
$xI$ is $\sigma$-conjugate under ${}^xI\cap I$ to $x$.

Now let $P=MN$ be a semistandard parabolic subgroup of $G$. There is then  an
Iwahori decomposition $I=I_NI_MI_{\overline N}$. We use the Iwahori subgroup $I_M$ of
$M(L)$ to form the subgroup $\Omega_M \subset
\widetilde W_M$; note that the canonical surjective homomorphism $\widetilde
W_M
\twoheadrightarrow \Lambda_M$ restricts to an isomorphism $\Omega_M \cong
\Lambda_M$. We compose this isomorphism with the canonical homomorphism
$\Lambda_M \to \mathfrak a_M$, obtaining a homomorphism $\Omega_M \to
\mathfrak a_M$; for $x \in \Omega_M$ we will denote by $\nu_x\in\mathfrak
a_M$ the image of
$x$ under this homomorphism. Note that $x \mapsto \nu_x$ is intrinsic to $M$
and has nothing to do with $P$.

\begin{defn} \label{fundamental-P-alcove}
For $x \in \widetilde W_M$ we say that $x\mathbf a$ is a \emph{fundamental
$P$-alcove} if it is a $P$-alcove for which $x \in \Omega_M$, or, in other
words, if $xI_Mx^{-1}=I_M$, $xI_Nx^{-1}\subset I_N$, and $x^{-1}I_{\overline N}x
\subset I_{\overline N}$.
\end{defn}
Proposition \ref{sigma_conj_prop} implies that any fundamental $P$-alcove is a
fundamental alcove, just as the terminology suggests. An obvious question
(that we have not tried to answer) is whether any fundamental alcove arises
as a fundamental $P$-alcove for some semistandard $P$.

The next result gives some insight into $P$-alcoves, although we will make
only incidental use of it. We write $\rho_N \in \mathfrak a^*$ for the
half-sum of the elements in $R_N$.
\begin{prop}
Write $\Omega_P$ for the set of $x \in \Omega_M$ such that $x\mathbf a$ is a
fundamental $P$-alcove.
\begin{enumerate}
\item $\Omega_P$ is a submonoid of $\Omega_M$.
\item Let $x,y \in \Omega_P$. Then $IxIyI=IxyI$ and
$\ell(x)+\ell(y)=\ell(xy)$. Here $\ell$ is the usual length function on
$\widetilde W$.
\item Let $x \in \Omega_P$. Then $\ell(x)=\langle 2\rho_N,\nu_x
\rangle$.
\end{enumerate}
\end{prop}
\begin{proof}
(1) This is clear from the definitions.

(2) For the first statement just note that
\[
xIy=(xI_Nx^{-1})xy(y^{-1}I_My)(y^{-1}I_{\overline N}y) \subset I_NxyI_MI_{\overline N} \subset
IxyI.
\]
The second statement follows from the first (easy, and presumably well-known).

(3) Since both the left and right sides of the equality to be proved are
additive functions on the monoid $\Omega_P$, we may replace $x$ by
$x^m$ for any positive integer $m$. Taking $m$ to be the order of the image
of $x$ in $W_M$, we are reduced to the case in which $x$ is a translation
element lying in
$\Omega_P$. Such an element is of the form $\epsilon^\mu$ for some
 cocharacter
$\mu
\in X_*(A)$ whose image is central in $M$ and  dominant with respect to any
Borel subgroup of $P$ containing $A$. It is easy to see that
$\nu_x$ is simply the image of
$\mu$ under the canonical inclusion of $X_*(A)$ in
$\mathfrak a$. Thus the equality to be proved is a consequence of the
equality $\ell(\epsilon^\mu)=\langle 2\rho_N,\mu \rangle$, which in turn
follows from the usual formula for the length of translation elements in
$\widetilde W$, in view of the fact that all roots of $M$ vanish on $\mu$.
\end{proof}

\subsection{Levi subgroups adapted to $I$} Let $M$ be a Levi subgroup of $G$
containing $A$. Once again we put $I_M=M(L)\cap I$ and form $\Omega_M
\subset \widetilde W_M$ relative to $I_M$. We will also make use of the
homomorphism $x \mapsto \nu_x$ from $\Omega_M$ to $\mathfrak a_M$ that was
explained in the previous subsection.

We write $\mathcal P(M)$ for the set of parabolic subgroups of $G$ having
$M$ as Levi component. For $P \in \mathcal P(M)$ we define $\Omega_M^{\ge
0}$ (respectively, $\Omega_M^{> 0}$) to be the set of elements $x \in \Omega_M$ such
that
$\langle \alpha,\nu_x \rangle \ge 0$ (respectively, $\langle \alpha,\nu_x \rangle >
0$) for all $\alpha \in R_N$. It is clear that most elements of  $\Omega_M^{\ge 0}$
lie in $\Omega_P$; however, we are going to give a condition on $M$ which will
guarantee that every element of $\Omega_M^{\ge 0}$ lies in $\Omega_P$.  (Compare this with Remark  \ref{rem.minu}, which
shows that when $P = MN$ is standard, an element $\epsilon^\lambda w \in \Omega_M$ lies in $\Omega_P$ if and only if $\lambda $ is $G$-dominant.)

As usual the group $\widetilde W_M$ acts by affine linear transformations on
both $\mathfrak a$ and its quotient $\mathfrak a/\mathfrak a_M$, the natural
surjection $\mathfrak a \twoheadrightarrow \mathfrak a/\mathfrak a_M$ being
$\widetilde W_M$-equivariant. The subgroup $\Omega_M$ then inherits an
action on $\mathfrak a$ and $\mathfrak a/\mathfrak a_M$.

\begin{defn}\label{def.adap}
We say that $M$ is \emph{adapted to} $I$ \textup{(}respectively, \emph{weakly adapted
to} $I$\textup{)} if there exists
$\lambda
\in
\mathbf a$ \textup{(}respectively, in the closure of $\mathbf a$\textup{)} whose
image in
$\mathfrak a/\mathfrak a_M$ is fixed by the action of
$\Omega_M$.
\end{defn}
For any such $\lambda$ it is easy to see that
$x\lambda=\lambda+\nu_x$ for all $x \in \Omega_M$.

\begin{prop}\label{prop.adap}
If $M$ is adapted to $I$, then $\Omega_M^{\ge 0} \subset
\Omega_P$, and consequently for every $x \in \Omega_M$ there exists $P
\in\mathcal P(M)$ for which $x\mathbf a$ is a fundamental $P$-alcove. Similarly, if
$M$ is weakly adapted to $I$, then $\Omega_M^{> 0} \subset \Omega_P$.
\end{prop}
\begin{proof} We begin by proving the first statement.
For $\alpha \in R_N$ we must show that $x\mathbf a \ge_\alpha \mathbf a$, which is
to say that
$k(\alpha,x\mathbf a)
\ge k(\alpha,\mathbf a)$.
For any $\lambda \in \mathbf a$ we have $k(\alpha,x\mathbf a)=\lceil
\alpha(x\lambda)\rceil$ and $k(\alpha,\mathbf a)=\lceil \alpha(\lambda)\rceil$.
 Now
pick $\lambda$ as in the definition of being adapted to
$I$.  Since $x \in
\Omega_M^{\ge0}$, we see from the equality
$x\lambda=\lambda+\nu_x$ that
$\alpha(x\lambda) \ge \alpha(\lambda)$; it is then clear that $\lceil
\alpha(x\lambda)\rceil\ge \lceil
\alpha(\lambda)\rceil$.

Now we prove the second statement. For $\alpha \in R_N$ we now have
\[
k(\alpha,\mathbf a)-1 \le \alpha(\lambda) < \alpha(x\lambda) \le k(\alpha,x\mathbf
a)
\]
and hence $k(\alpha,\mathbf a)  \le k(\alpha,x\mathbf a)$, as desired.
\end{proof}

\begin{prop}
Let $M$ be any Levi subgroup containing $A$. Then there exists $w \in W$
such that ${}^wM$ is adapted to $I$.
\end{prop}
\begin{proof}
There exist fixed points of $\Omega_M$ on $\mathfrak a/\mathfrak a_M$ lying on no
affine root hyperplane for $M$ (for example, when $M$ is simple, one can take the
barycenter of the base alcove for $\widetilde W_M$). We choose such a fixed point
$\overline\lambda$ and then choose  $\lambda \in \mathfrak a$ mapping to $\overline\lambda$.
We are free to add any element of $\mathfrak a_M$ to $\lambda$, so we may assume that
$\lambda$ lies on no
affine root hyperplane for $G$. If
$\lambda$ happens to lie in $\mathbf a$, then $M$ is adapted to $I$. In any case
there exists a unique alcove  $x'\mathbf a$ containing $\lambda$. The Levi
subgroup is then adapted to $I'=x'Ix'^{-1}$. Taking $w$ to be the inverse of the image
of $x'$ in $W$, we find that ${}^wM$ is adapted to $I$.
\end{proof}

Being adapted to $I$ is quite a strong condition on $M$. It is important to realize
that standard Levi subgroups are often not
adapted to our standard Iwahori subgroup $I$, even though both notions of
standard are tied to the same Borel subgroup.

\begin{cor}\label{cor.fundsup}
For every $[b] \in B(G)$ there exists a semistandard representative $x \in \widetilde
W$ of $[b]$ such that $x\mathbf a$ is a fundamental alcove and hence
$IxI \subset [b]$.
\end{cor}
\begin{proof}
This follows from the previous two propositions and Definition \ref{defn_std_repr}.
\end{proof}

\subsection{Superset method} Let $b \in G(L)$. The \emph{superset}
$\widetilde W(b)$ associated to $b$ is the set of $x \in \widetilde W$ such that
$IxI$ is contained in $Iy^{-1}IbIyI$ for some $y \in \widetilde W$. The reason for
the name superset is that the set of $x \in \widetilde W$ such that $X_x(b) \ne
\emptyset$ is contained in $\widetilde W(b)$. Indeed, if $X_x(b) \ne \emptyset$, then
there exists $g \in G(L)$ such that $g^{-1}b\sigma(g) \in IxI$. There also exists $y
\in \widetilde W$ such that $g \in IyI$, and then
\[
IxI=Ig^{-1}b\sigma(g)I\subset Iy^{-1}IbIyI.
\]

\begin{prop}\label{prop.b0fa}
Suppose that $x_0\mathbf a$ is a fundamental alcove, and
let $b_0$ be any element of $Ix_0I$. Then
\[
\{x \in \widetilde W: X_x(b_0)\ne\emptyset\}=\widetilde W(b_0).
\]
\end{prop}
\begin{proof}
We already know the inclusion $\subset$. To establish $\supset$ we consider $x \in
\widetilde W(b_0)$ and choose $y \in \widetilde W$ such that $IxI \subset
Iy^{-1}Ib_0IyI$. Then $IxI$ meets $y^{-1}Ib_0Iy$, and since (by our
hypothesis on $x_0$) every element of $Ib_0I$ has the form $i^{-1}b_0\sigma(i)$ for
suitable $i \in I$,  there is some element in $IxI$ of the form
$\dot{y}^{-1}i^{-1}b_0\sigma(i)\dot{y}$, where $\dot{y}$ is a representative of
$y$ in the $F$-points of the normalizer of $A$ in $G$. Since
$\dot{y}=\sigma(\dot{y})$, this shows that $IxI$ meets $[b_0]$, as desired.
\end{proof}

\begin{cor} \label{cor.superset_method}
For every $[b] \in B(G)$ there is a semistandard representative $b_0 \in [b]$ for
which the superset method applies, yielding
\[
\{x \in \widetilde W: X_x(b_0)\ne\emptyset\}=\widetilde W(b_0).
\]
\end{cor}
\begin{proof}
Combine Corollary \ref{cor.fundsup} with Proposition \ref{prop.b0fa}.
\end{proof}

\section{Examples}\label{sec.examples}
\subsection{}
To illustrate our results and conjectures (Conjecture~\ref{conj2} and Conjecture~\ref{conj3}~(a)), in this section we present two examples for the group $GSp_4$ (i.~e.~for Dynkin type $C_2$). In the first example, $b=1$, in the second one, $b$ is one of the generators of the subgroup $\Omega\subset\widetilde W$ of all length $0$ elements (the picture is independent of the choice of generator; in fact, it depends only on the parity of the image of $b$ under an isomorphism $\Omega\cong \mathbb Z$).

In both cases, we identify the coset $W_ab\subset\widetilde W$ with the set of alcoves in the standard apartment. Here, the origin is marked by a dot, and the base alcove is black. Gray alcoves correspond to non-empty affine Deligne-Lusztig varieties (and the number given is the dimension), while white alcoves correspond to empty ones.

The thick black lines indicate the shrunken Weyl chambers. The dashed lines indicate the $W$-cosets $\varepsilon^\mu W$ inside the shrunken Weyl chambers. Recall the maps $\eta_1$ and $\eta_2$ from Section~\ref{relation_with_Reumans_conj}: Viewing each dashed square as a copy of the finite Weyl group, $\eta_1$ maps an element to the position it has inside the dashed square it lies in (i.e., to the corresponding element of $W$). On the other hand, the map $\eta_2$ is constant on each finite Weyl chamber, i.e., it maps an alcove to the finite Weyl chamber it lies in, considered as an element of $W$. As the conjecture predicts, inside a shrunken Weyl chamber all dashed squares look the same (independently of $b$!).

For further examples, we refer to \cite{GHKR}, and also to the version of that paper on the arxiv server (arXiv:math/0504443v1).

\newpage
\begin{figure}[h!]
\includegraphics[width=12cm]{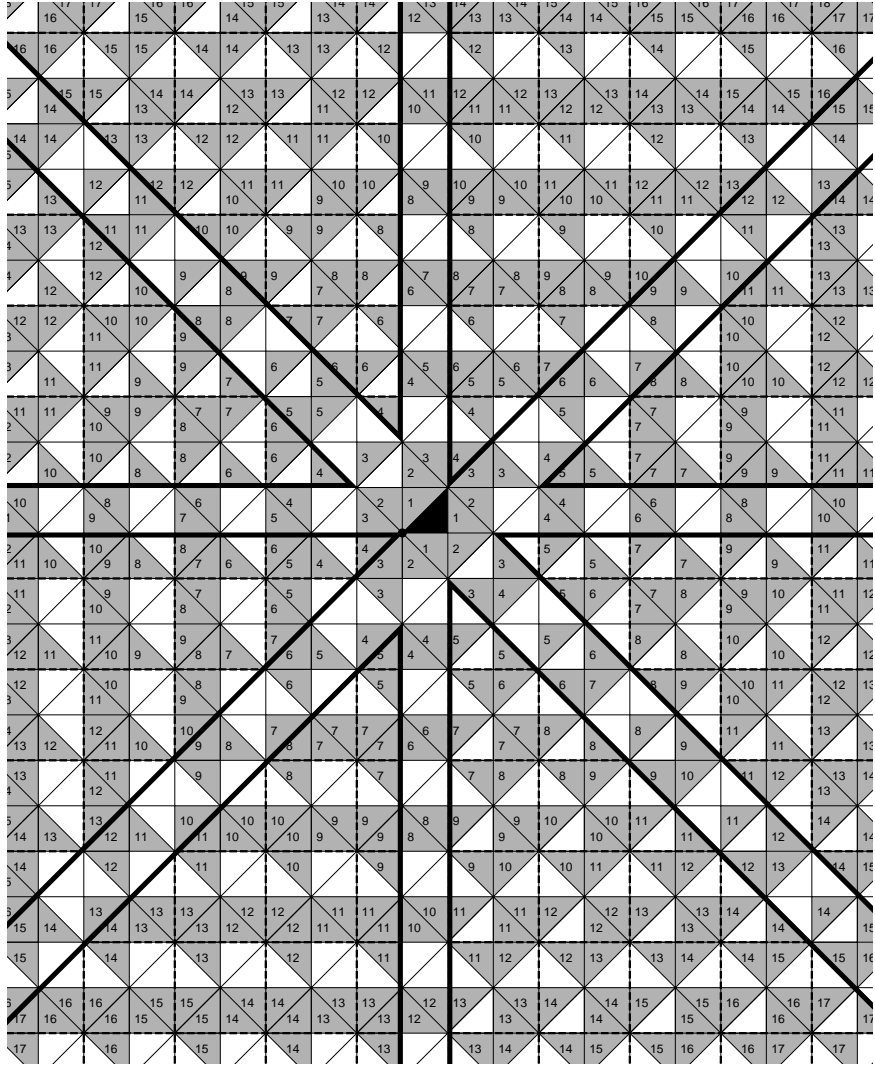}
\caption{Dimensions of ADLV for type $C_2$, $b=1$.}
\label{C2-1}
\end{figure}
\newpage
\begin{figure}[h!]
\includegraphics[width=12cm]{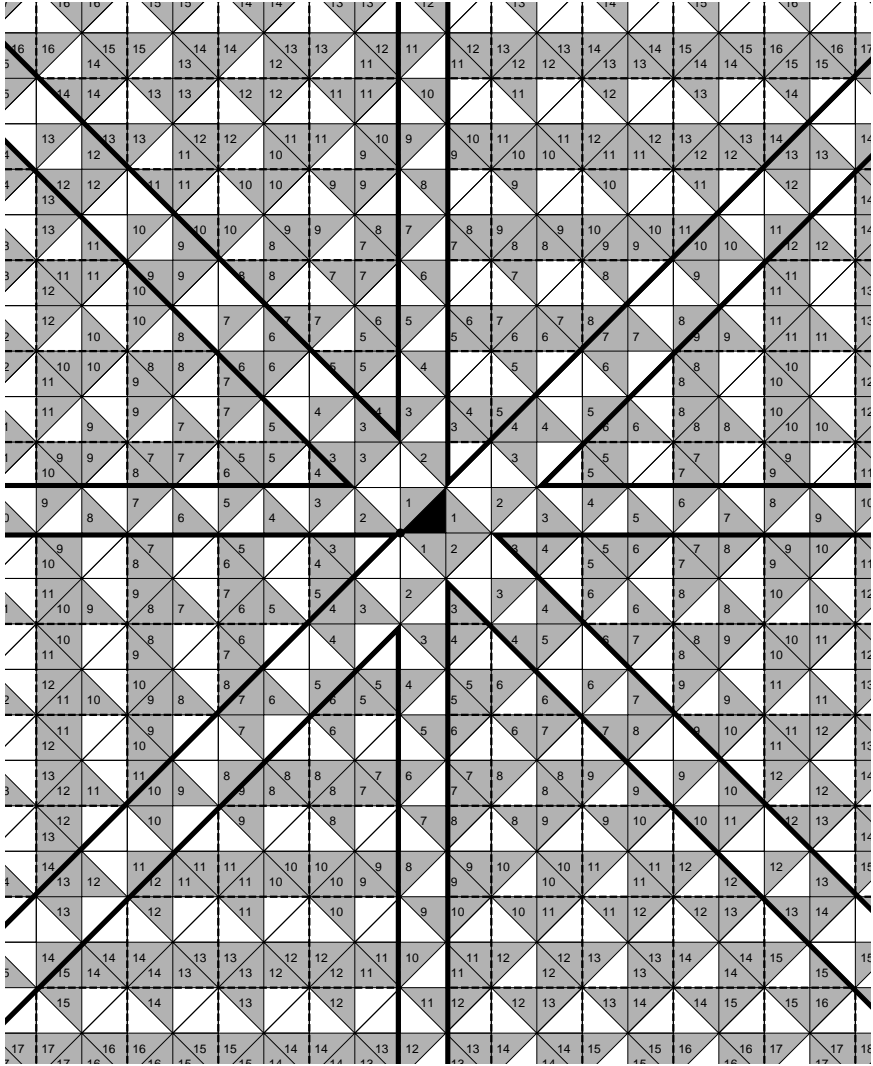}
\caption{Dimensions of ADLV for type $C_2$, $b$ ``supersingular''.}
\end{figure}

\end{document}